\documentclass[12pt]{article}
\usepackage[utf8]{inputenc}
\usepackage{amsmath,amssymb,amsthm,tikz-cd,comment,multirow,mathrsfs}
\PassOptionsToPackage{hyphens}{url}
\usepackage{hyperref}
\usepackage[normalem]{ulem}
\usepackage{cite}

\usepackage[left=1in,right=1in,top=1in,bottom=1in]{geometry}

\newtheoremstyle{Normal}{}{}{}{}{\bfseries}{:}{.5em}{}
\theoremstyle{Normal}
\newtheorem{theorem}{Theorem}[section]
\newtheorem{lemma}[theorem]{Lemma}
\newtheorem{prop}[theorem]{Proposition}
\newtheorem{example}[theorem]{Example}

\newtheorem{defn}[theorem]{Definition}

\newtheorem{remark}[theorem]{Remark}

\newcommand{\EE}{\mathbb E}
\newcommand{\II}{\mathbb I}
\newcommand{\NN}{\mathbb N}
\newcommand{\PP}{\mathbb P}
\newcommand{\RR}{\mathbb R}
\newcommand{\XX}{\mathbb X}
\newcommand{\ZZ}{\mathbb Z}

\newcommand{\Sp}{\mathcal S}
\newcommand{\C}{\mathcal C}
\newcommand{\Reac}{\mathcal R}
\newcommand{\Rate}{\mathcal K}
\newcommand{\Fsim}{F^{\mathrm{sim}}}
\newcommand{\compart}{\mathcal H}
\newcommand{\chem}{\mathcal I}
\newcommand{\whole}{\mathcal F}
\newcommand{\CoarseStateSpace}{\mathcal N}

\newcommand{\norm}[1]{\left\lVert#1\right\rVert}

\newcommand{\MIN}{\wedge}

\title{Stochastic reaction networks within interacting compartments}
\author{David F.~Anderson\thanks{Department of Mathematics, University of
  Wisconsin-Madison,  anderson@math.wisc.edu.} \and Aidan S.~Howells\thanks{Department of Mathematics, University of
  Wisconsin-Madison,  ahowells@wisc.edu. Corresponding author.}}

\begin{document}

\maketitle


\begin{abstract}
Stochastic reaction networks, which are usually modeled as continuous-time Markov chains on $\ZZ^d_{\ge 0}$, and simulated via a version of the ``Gillespie algorithm,'' have proven to be a useful tool for the understanding of processes, chemical and otherwise, in homogeneous environments. There are multiple avenues for generalizing away from the assumption that the environment is homogeneous, with the proper modeling choice dependent upon the context of the problem being considered. One such generalization was recently introduced in \cite{Duso_Zechner_2020}, where the proposed model includes a varying number of interacting compartments, or cells, each of which contains an evolving copy of the stochastic reaction system. The novelty of the model is that these compartments also interact via the merging of two compartments (including their contents), the splitting of one compartment into two, and the appearance and destruction of compartments. In this paper we begin a systematic exploration of the mathematical properties of this model.  We (i) obtain basic/foundational results pertaining to explosivity, transience, recurrence, and positive recurrence of the model, (ii) explore a number of examples demonstrating some possible non-intuitive behaviors of the model, and (iii) identify the limiting distribution of the model in a special case that generalizes  three formulas from an example in \cite{Duso_Zechner_2020}.
\end{abstract}

\section{Introduction}
Stochastic reaction networks are now commonly utilized to model various types of systems in the biological sciences.  These mathematical models are often continuous-time Markov chains and are used when the counts of at least some of the underlying ``species,'' which are most commonly different molecule types, are low.  In this low copy-number case, the state of the model is a vector giving the integer counts of the different species and transitions are governed by the different possible ``reactions'' that can take place.  These models are typically simulated via the Gillespie algorithm \cite{Gillespie_1976,Gillespie_1977} or the next reaction method \cite{Gibson_Bruck_2000,Anderson_2007}.  See \cite{Anderson_Kurtz_2015}, and references therein, for more on this type of model.

One potential drawback to the standard model is that it assumes a homogeneous environment.  There are multiple ways to generalize, however.  One common generalization is to split the state space itself into differed fixed pieces (often called ``voxels'') and then allow for transitions between adjacent voxels \cite{Isaacson_2013,Popovic_McKinley_Reed_2011}. Thinking of the size of the voxels going to zero leads naturally to a model with continuous space in which the state of the system is given by the type, position, velocity, etc. of each particle in the system.  A reaction can then only take place when the necessary constituent molecules are near each other (with the precise mechanism for defining when they are ``near enough'' left to the modeler). One of the first examples of such a continuous space model was introduced by Doi \cite{Doi_1976}. More generally, there are a whole class of continuous space models known as reaction-diffusion models.  For a brief overview of such models, see \cite{Erban_Othmer_2014}. For a comparison of two specific such models, with an approachable introduction, see \cite{Agbanusi_Isaacson_2014}; for a more general approach, see the introduction of \cite{Razo_Winkelmann_Klein_Hofling_2023}.

A different approach to generalize from the homogeneous case  is to imagine some fixed collection of compartments and model the dynamics within each compartment in the usual way (as a continuous-time Markov chain as described in the first paragraph above) while also  allowing for interactions between adjacent compartments. This is the approach taken in \cite{McKane_Newman_2004} in an ecological context (their ``patches" are our ``compartments"). 
However, ideally one might like to also account for situations like in biological tissue, where reactions take place in cells that are not static but, for example, can appear, divide, possibly merge, or even be destroyed. That is the approach presented in a recent paper by Duso and Zechner, where they developed a Markov model for stochastic reaction networks within interacting compartments \cite{Duso_Zechner_2020}. In particular, their model consists of two basic components:
\begin{enumerate}
    \item a stochastic model of a chemical reaction network;  
    \item a dynamic model of compartments, or cells, which themselves undergo basic transitions such as (i) arrivals, (ii) departures, (iii) mergers, and (iv) divisions.     In the context of \cite{Duso_Zechner_2020}, these four transition types are referred to as inflows, exits, coagulations, and fragmentations, respectively.
\end{enumerate}
Each compartment, or cell, contains a copy of the (evolving) chemical reaction network. When two cells merge, their contents are combined.  When a cell divides, its contents are randomly split among the two new daughter cells.  
Beyond the framework itself, their paper focuses on the framework's practical use, using moment closure methods to derive estimates for various population statistics which are then validated by simulation. They also derive stationary distributions for some special cases.

In the present paper, we attempt to lay the groundwork for exploration of mathematical questions about the Markov chain model developed in \cite{Duso_Zechner_2020}.
We focus on the special case where the compartments can only enter, leave, merge, and divide, all according to mass action kinetics and unaffected by their contents.
Questions pertaining to recurrence, transience, and explosivity are all considered.  We show that in most, but not all, parameter regimes the overall qualitative behavior of the model (i.e., recurrence, transience, or explosivity) is the same as that of the associated stochastic reaction network. We also analyze myriad examples that, taken together, demonstrate some of the non-intuitive (and interesting) possible behaviors of the model.
Moreover, we derive the stationary distribution for the model in the case where the chemistry inside the compartments is well understood in the sense that a formula for the distribution is known for all time (e.g., the DR models of \cite{Anderson_Schnoerr_Yuan_2020}) and the compartments themselves are not allowed to interact (but are not totally static, being allowed to enter and leave the system). Two special cases of this stationary distribution are provided as illustration, both of which generalize formulas from an example in \cite{Duso_Zechner_2020}.

Before moving on, we warn the reader that in the field of epidemiology, the term ``compartment model" has a different meaning. There the compartments are what we would call species. For example, they would speak of an SIR model as dividing individuals into a susceptible compartment, an infected compartment, and a recovered compartment. See e.g.~\cite{Brauer_2008}.

A standard knowledge of continuous-time Markov chains is assumed. See for example Norris \cite{Norris_1997} for a  detailed introduction to the topic. 
For notational convenience, we will use the following shorthand notations: for any two vectors $v,w \in \RR^d_{\ge 0}$ and any vector $x,y\in \ZZ^d_{\ge 0}$ we denote
\begin{align*}
v^w = \prod_{i=1}^d (v_i)^{w_i} \quad \text{and}
\quad x!= \prod_{i=1}^d(x_i)!\quad\text{and}
\quad \binom xy=\prod_{i=1}^d\binom{x_i}{y_i},
\end{align*}
with the conventions that $0^0=1$ and that $\binom xy=0$ for $y<0$ or $y>x$. Moreover, we will always use $d$ to represent the number of species in the model.  Finally, for $x \in \ZZ^d_{\ge 0}$  we define $e_x: \ZZ^d_{\ge 0} \to \ZZ$ to be the function taking the value of one at $x$ and zero otherwise.

The remainder of the paper is outlined as follows. In section \ref{sec:model}, we fully specify the model.  Further, we give two different mathematical representations that are both useful and prove some first basic properties.  In the brief section \ref{sec:explosivity}, we prove that the full model is explosive if and only if the associated reaction network is.  In section \ref{sec:recurrenc}, we give conditions for when the full model is recurrent, positive recurrent, or transient. 
Finally, in section \ref{sec:SD}, we provide the stationary distribution for a special class of models.

\section{The reaction network within interacting compartments (RNIC) model}
\label{sec:model}

As discussed in the introduction, the full model we consider here consists of two sub-models: (i) a stochastic reaction network and (ii) a dynamic model of compartments, or cells, each of which contains an evolving copy of the stochastic reaction network.  We first describe these sub-models individually and then specify how they are combined to make the full model.

\subsection{Stochastic reaction networks}
Suppose we have a finite set $\Sp$, whose elements we shall call \emph{species}, and a directed graph whose vertices are unique linear combinations of species with non-negative integer coefficients. The edges of the graph are called \emph{reactions}; let $\Reac$ denote the set of reactions. The linear combinations which appear as vertices in the graph are called \emph{complexes}; the set of complexes will be denoted $\C$. A \emph{chemical reaction network} (or just \emph{reaction network}; \emph{CRN} for short) is the tuple $\chem = (\Sp,\C,\Reac)$, where $\Sp$, $\C$ and $\Reac$ are as above. See Figure \ref{fig:240895728} for an example reaction network.

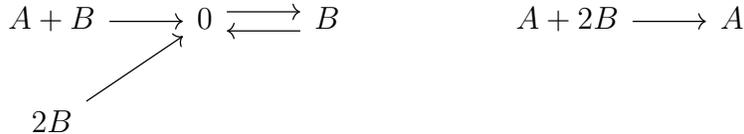
\begin{figure}
\begin{center}
\begin{tikzcd}
   A+B\arrow{r}&0\arrow[yshift=.7ex]{r}& B\arrow[yshift=-.7ex]{l}&&A+2B\arrow{r}&A\\
    2B\arrow{ur}
\end{tikzcd}
\end{center}
\caption{
The CRN with species $A$ and $B$ and reactions $A+B\to0$, $0\to B$, $B\to 0$, $2B\to 0$, and $A+2B\to A$.
Note that $0$ here denotes the linear combination $0A+0B$.}
\label{fig:240895728}
\end{figure}

When talking about specific reaction networks, the species will usually be represented by capital Latin letters. When talking generally, there will be $d$ species $S_1,\dots,S_d$. In this case we will identify $\ZZ^d$ with the space of linear combinations of species with integer coefficients. That is, we naturally identify $\nu\in \C$ with the vector in $\ZZ^d$ whose $i$th element is the coefficient of $S_i$ in $\nu$. We will speak of reactions $\nu\to\nu'\in\Reac$, or sometimes, when we wish to enumerate the reactions as $\{\nu_r \to \nu_r'\}$, we will simply write $r\in \Reac$.

There are multiple ways to associate a mathematical model to a given reaction network, including the use of a deterministic ODE \cite{Shinar_Feinberg_2010}, a diffusion process \cite{Leite_Williams_2019, Anderson_Higham_Leite_Williams_2019}, and a continuous-time Markov chain \cite{Anderson_Kurtz_2015}. The only one of concern to us here is the continuous-time Markov chain model with stochastic mass-action kinetics, in which the state of the system is a vector giving the number of each species present and transitions are determined by the reactions. To fully specify the model, positive (or sometimes, merely non-negative) numbers, called \textit{rate constants}, are assigned to each reaction. If the reaction $\nu\to \nu'$ has rate constant $\kappa$, then in state $x$ that particular reaction occurs with rate $\kappa \binom{x}{\nu}$ and when it occurs the chain transitions to state $x+\nu'-\nu$. So the reactions will happen with rate proportional to the number of ways the chemicals can combine to allow them to happen, and $\kappa$ is the constant of proportionality. If $\Rate$ is a set of rate constants, one for each reaction, we denote by $\chem_{\Rate} = (\Sp,\C,\Reac,\Rate)$ the corresponding stochastic mass-action system. If we let $\kappa_{\nu\to \nu'}$ be the rate constant for the reaction $\nu \to \nu'$, then the Markov chain transitions from state $x\in \ZZ^d_{\ge0}$ to state $y\in \ZZ^d_{\ge 0}$ with rate
\begin{align}\label{eq:massaction}
   q(x,y) = \mathop{\sum_{\nu\to \nu'\in\Reac}}_{\nu'-\nu=y-x}\kappa_{\nu\to \nu'} \binom{x}{\nu}
   =\mathop{\sum_{\nu\to \nu'\in\Reac}}_{\nu'-\nu=y-x}\kappa_{\nu\to \nu'}\prod_{j=1}^d\binom{x_j}{\nu_j}
\end{align}
where the sum is over those reactions for which $\nu'-\nu=y-x$. For $r=\nu_r\to\nu'_r\in\Reac$, we denote the rate of the reaction $r$ in state $x\in\ZZ_{\ge0}^d$ by $\lambda_r(x)$:
\begin{align}\label{eq:lambda}
    \lambda_r(x)=\kappa_{r} \binom{x}{\nu_r}
\end{align}
Note that $\lambda_{\nu\to\nu'}(x)=0$ if $x_i < \nu_i$ for some $i$, since $\binom mk=0$ for $k>m$. Note also that not all authors take the same conventions as we do here. In fact, the convention we use here pertaining to our rate constants is more in line with the biology literature \cite{wilkinson2018stochastic}.  In the mathematical literature it is more common to use a falling factorial $\lambda_{\nu\to \nu'}(x)=\kappa_{\nu\to \nu'}\prod_j (x_j)(x_j-1)\cdots(x_j-\nu_j+1) = \kappa_{\nu\to \nu'} \frac{x!}{(x-\nu)!}$, at the cost that their rate constant $\kappa$ is no longer the constant of proportionality when the reaction takes multiple inputs \cite{Anderson_Kurtz_2015}. This choice plays no fundamental role in our results, but makes certain expressions cleaner in the present context.

We note here that many of the results found in this paper can be generalized to systems with kinetics, i.e., rate functions $\lambda_r$, other than mass-action.  See Remark \ref{rmk:PR_NMA}.

Put more succinctly, we have a Markov process on $\ZZ^d_{\ge 0}$ with infinitesimal generator 
\[
    \mathcal{L}f(x) = \sum_{r \in \Reac} \lambda_{r}(x) (f(x+\nu_r' - \nu_r)-f(x)),
\]
where $\lambda_r$ is determined via \eqref{eq:lambda}, and the above is valid for all functions $f$ that are compactly supported \cite{Ethier_Kurtz_2009}.  The Kolmogorov forward equation, often called the chemical master equation in the context of reaction networks, is then
\begin{align*}
\frac{d}{dt} P_\mu(x,t) = \sum_{r\in \Reac} \lambda_r(x-(\nu_r'-\nu_r)) P_\mu(x-(\nu_r'-\nu_r),t) -  \sum_{r\in \Reac} \lambda_r(x) P_\mu(x,t),
\end{align*}
where $P_\mu(x,t) = P_\mu(X(t) = x)$ is the probability the process $X$ is in state $x\in\ZZ^d_{\ge 0}$ at time $t$, given an initial distribution of $\mu$. We take the convention that $P_\mu(x,t) = 0$ for $x\notin\ZZ^d_{\ge 0}$.

One way to represent the solution to the stochastic model described above is via a representation developed and popularized by Thomas Kurtz. Let $\{Y_{r}\}_{r\in \Reac}$ be a collection of independent, unit-rate Poisson processes, one for each possible reaction, and let $X(t), t\ge 0$, be the solution to
\begin{align}\label{eq:RTC}
    X(t) = X(0) + \sum_{r \in \Reac} Y_r \left( \int_0^t \lambda_r(X(s)) ds\right) (\nu_r'-\nu_r),
\end{align}
then $X$ is a continuous-time Markov chain that satisfies the conditions of the model specified above \cite{Anderson_Kurtz_2015,Ethier_Kurtz_2009,Kurtz_1978}.

\begin{example}\label{example:09878097}
Suppose we assign rate constants to the  example CRN in Figure \ref{fig:240895728} as follows:
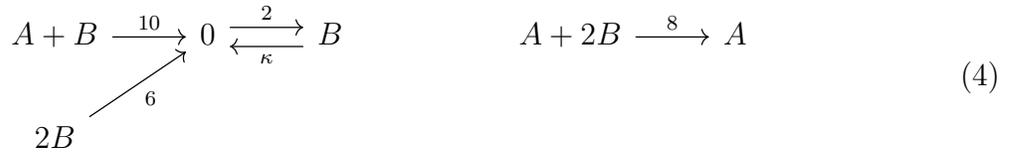
\begin{equation}\label{eq:08970897}
\begin{tikzcd}
   A+B\arrow{r}{10}&0\arrow[yshift=.7ex]{r}{2}& B\arrow[yshift=-.7ex]{l}{\kappa}&&A+2B\arrow{r}{8}&A\\
    2B\arrow[swap]{ur}{6}
\end{tikzcd}
\end{equation}
Let $x=(a,b)\in\ZZ_{\ge0}^2$ denote an arbitrary state of the system.  For the particular choice of rate constants given above the positive transition rates $q((a,b),\cdot)$, for $a, b \in \ZZ_{\ge 0}$,  are
\begin{center}
\begin{tabular}{ccccccl}
Reaction(s) &&& Transition &&& Rate\\\hline
$A+B\to0$&&& $(a,b)\mapsto(a-1,b-1)$ &&& $10ab$\\[0.5ex]
$0\to B$&&& $(a,b)\mapsto(a,b+1)$ &&& $2$\\[0.5ex]
$2B\to0$ and $A+2B\to A$&&& $(a,b)\mapsto(a,b-2)$ &&& $\displaystyle 6\frac{b(b-1)}{2}+8a\frac{b(b-1)}{2}$\\[0.5ex]
$B\to0$&&& $(a,b)\mapsto(a,b-1)$ &&& $\kappa b$
\end{tabular}
\end{center}
We chose to write $6\frac{b(b-1)}{2}+8a\frac{b(b-1)}{2}$ instead of $3b(b-1)+4ab(b-1)$ to emphasize our choice of intensity functions. Note that all other rates, such as $q((a,b),(a+1,b))$ or $q((a,b),(a+12,b-3))$, are zero.\hfill $\triangle$
\end{example}

\subsection{Compartment model}
\label{sec:compartment}

Having fully specified our CRN, $\chem_{\Rate} = (\Sp,\C,\Reac,\Rate)$, we turn to our next sub-model: the compartment model. As mentioned in the introduction, we will assume that compartments, or cells, can arrive, depart, merge, and divide. We can use the notation of chemical reaction networks to describe the four possibilities visually via a reaction network,
\begin{equation*}
\begin{tikzcd}
    0\arrow[yshift=.7ex]{r}& C\arrow[yshift=-.7ex]{l}
    \arrow[yshift=.7ex]{r}&
    2C\arrow[yshift=-.7ex]{l}
\end{tikzcd}
\end{equation*}
with $0\to C$ representing arrivals, $C\to 0$ representing departures, $C \to 2C$ representing division, or fragmentation, and $2C \to C$ representing mergers, or coagulations. Moreover, we assume that the stochastic model tracking the number of compartments behaves as a standard stochastic reaction network as already described in the previous section (however, see Remark \ref{rmk:PR_NMA} for an allowable generalization to the choice of kinetics). We will term this reaction network the \textit{compartment network}, and denote it by $\compart=(\Sp_{\text{comp}},\C_{\text{comp}},\Reac_{\text{comp}})$. Note that $\Sp_{\text{comp}} = \{C\}$ and $\C_{\text{comp}}$ is a subset of $\{0,C,2C\}$ (depending on which rate constants are non-zero). If rate constants are added as follows,
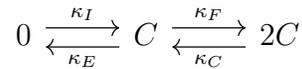
\begin{equation*}
\begin{tikzcd}
    0\arrow[yshift=.7ex]{r}{\kappa_I}& C\arrow[yshift=-.7ex]{l}{\kappa_E}
    \arrow[yshift=.7ex]{r}{\kappa_F}&
    2C\arrow[yshift=-.7ex]{l}{\kappa_C}
\end{tikzcd}
\end{equation*}
where each $\kappa_E, \kappa_I, \kappa_C, \kappa_F\ge 0$, then we will denote the corresponding stochastic mass-action system by $\compart_\Rate=(\Sp_{\text{comp}},\C_{\text{comp}},\Reac_{\text{comp}},\Rate_{\text{comp}})$.
According to \eqref{eq:RTC}, if we denote by $M_C(t)$  the number of compartments at time $t$, then  one way to represent this model is as the solution to
\begin{align*}
M_C(t) &= M_C(0) + Y_I\left(  \kappa_I t\right)- Y_E\left(  \int_0^t \kappa_E M_C(s) ds\right)+ Y_F\left(  \int_0^t \kappa_F M_C(s) ds\right)\\
&\hspace{.5in}- Y_C\left(  \int_0^t \kappa_C \frac{M_C(s) (M_C(s) - 1)}{2}\, ds\right),
\end{align*}
where  $Y_I, Y_E, Y_F$, and $Y_C$ are independent unit-rate Poisson processes. 

\subsection{Specifying the full, combined model}

Our full model, which we will term a \textit{reaction network within interacting compartments} (RNIC), begins with two networks, one representing the dynamics of the compartments themselves and one representing the chemistry taking place inside the compartments.
\begin{itemize}
    \item A CRN $\compart_\Rate$ of the form $0\leftrightarrows C\leftrightarrows 2C$, called the \textit{compartment network}. The state of this CRN (in $\ZZ_{\ge0}$) will be the number of compartments.

    \item An CRN $\chem_\Rate$, called the \textit{chemistry} (or \emph{\underline Internal network}), with $d$ species.
\end{itemize}

The behavior of the model between transitions of the compartment model is straightforward: the CRN within each compartment evolves independently as a Markov chain with transition rates specified by \eqref{eq:massaction}.  All that remains is to specify what happens to the full model at the transition times of the compartment model. Hence, there are four cases to consider.

\begin{itemize}

    \item An arrival:  $0\to C$. We assume the existence of a probability measure $\mu$ on $\ZZ^d_{\ge 0}$. Each time  an arrival event occurs, we add a new compartment whose initial state is chosen according to $\mu$, independent of the past. (Note that $\mu$ is not necessary when $\kappa_I=0$.)
    
    \item A departure: $C\to 0$. When a departure event occurs, we choose one of the compartments, uniformly at random, for deletion.
        
    \item A merger: $2C \to C$. When a merger event occurs, we select two compartments, uniformly at random. We replace the chosen compartments with a single compartment. The state of the new compartment is the sum of the states of the two it replaced.
    
    \item A division: $C \to 2C$. When a division event occurs, we select a compartment, uniformly at random. We replace the  chosen compartment  with two new compartments, whose initial states are determined by having each molecule from the chosen compartment select one of the two new compartments uniformly. For example, if there are $n_A$ type $A$ species in the chosen compartment, then one of the new compartments will get a number of $A$ molecules given by a binomial distribution with parameters $n_A$ and $p = \tfrac12$, and the other compartment will get $n_A$ minus that value.
\end{itemize}
This whole system will be denoted $\whole=(\chem_\Rate,\compart_\Rate,\mu)$.

\begin{remark}
    Above, we assume that when divisions, i.e., compartment transitions of the form $C\to2C$, happen, each molecule picks a new compartment uniformly at random. This assumption makes the constructions in this paper easier.  However, our proofs only require that the total number  of each species across compartments is preserved when each division happens. 
\end{remark}

Similar to our network representations for reaction networks, we can specify the above model through a picture of the following form:
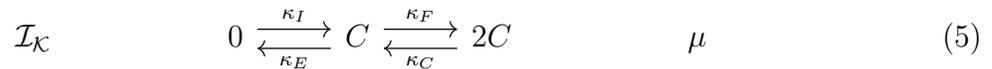
\begin{equation}\label{model}
\begin{tikzcd}
    \chem_\Rate &&
    0\arrow[yshift=.7ex]{r}{\kappa_I}& C\arrow[yshift=-.7ex]{l}{\kappa_E}
    \arrow[yshift=.7ex]{r}{\kappa_F}&
    2C\arrow[yshift=-.7ex]{l}{\kappa_C}&&
    \mu
\end{tikzcd}
\end{equation}
where ``$\chem_\Rate$'' is a stand-in for a standard CRN diagram, such as the one in \eqref{eq:08970897}.

\begin{example}\label{example:full_model}
If $\chem_\Rate$ is exactly the network diagrammed in Example \ref{example:09878097} and $\mu$ is the point mass with $3$ molecules of $A$ and $4$ molecules of $B$, we would write
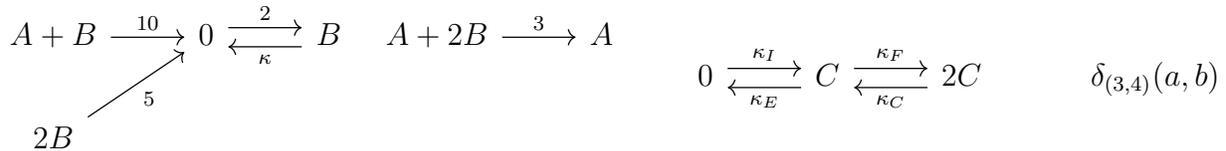
\begin{equation*}
\begin{tikzcd}
   A+B\arrow{r}{10}&0\arrow[yshift=.7ex]{r}{2}& B \arrow[yshift=-.7ex]{l}{\kappa} & \hspace{-.3in} A+2B\arrow{r}{3}&A\\
    2B\arrow[swap]{ur}{5}
\end{tikzcd}\hspace{0.3in} 
\begin{tikzcd}
    0\arrow[yshift=.7ex]{r}{\kappa_I}& C\arrow[yshift=-.7ex]{l}{\kappa_E}
    \arrow[yshift=.7ex]{r}{\kappa_F}&
    2C\arrow[yshift=-.7ex]{l}{\kappa_C}
\end{tikzcd}
\hspace{0.5in} \delta_{(3,4)}(a,b)
\end{equation*}
\hfill $\triangle$
\end{example}

See also Example \ref{Ex: Intro compartment model} for another specific example.

There are multiple avenues for generalizations. For example, when a merger occurs it could be that not all the molecules make it into the new compartment, or when a division occurs it could be that some molecules are lost, or there is a non-uniform mechanism for distributing the molecules.  Moreover, it could be that the rate of compartment fragmentation or exit depends on the internal state of the compartment. These models all fall under the more general framework given in \cite{Duso_Zechner_2020} and could be studied mathematically in the future if there is a desire, but for the initial development of the mathematics we choose to keep things simpler.

\subsubsection{Simulation representation}
\label{sec:simulation}
There are multiple ways to describe a Markov model satisfying the information given in the ingredients $\whole=(\chem_\Rate,\compart_\Rate,\mu)$. The first we give is what we term a ``simulation'' representation in which we enumerate the compartments and track the counts of the species in each compartment.

The simulation representation will be a Markov chain $\Fsim$ whose state is a finite vector of elements of $\ZZ^d_{\ge 0}$, where $d$, as always, is the number of species. We first describe the model via an example.  Afterwards we will provide the mathematical details.

\begin{example}\label{ex:simulation representation}
  Consider again the model from Example \ref{example:full_model}.  
  Suppose that at time $T$ there are 4 compartments, where the first has two $A$ and two $B$, the second has no $A$ and one $B$, the third again has two of each, and the last has one $A$ and twelve $B$. Then the state of the model $\Fsim$ would be the vector
  \begin{align*}
  \left( \left[ \begin{array}{c} 2\\2\end{array} \right], \left[ \begin{array}{c} 0\\1\end{array} \right], \left[ \begin{array}{c} 2\\2\end{array} \right], \left[ \begin{array}{c} 1\\12\end{array} \right]\right).
  \end{align*}
  We now suppose that at time $T$ a transition occurs.  We first consider four possibilities if the transition is due to a reaction of the compartment model.
  
  \begin{itemize}
  \item Suppose first that the compartment transition is an inflow event.  We will make the convention that the new compartment due to an inflow reaction will always be placed at the end of the vector of states. Hence, because  the initial distribution for arriving compartments is a point mass at $(3,4)$ the new state of the full system is
  \begin{align*}
  \left( \left[ \begin{array}{c} 2\\2\end{array} \right], \left[ \begin{array}{c} 0\\1\end{array} \right], \left[ \begin{array}{c} 2\\2\end{array} \right], \left[ \begin{array}{c} 1\\12\end{array} \right], \left[ \begin{array}{c} 3\\4\end{array} \right]\right).
  \end{align*}

  \item Next suppose that the compartment transition is an exit event.  In this case we must choose a compartment at random, delete it from the vector, and re-index the other components.  Thus, we start by choosing from $\{1,2,3,4\}$, each with probability $1/4$. Suppose that the value $3$ is chosen so that the third compartment will be deleted.  In this case,  the new state of the full system is
  \begin{align*}
  \left( \left[ \begin{array}{c} 2\\2\end{array} \right], \left[ \begin{array}{c} 0\\1\end{array} \right], \left[ \begin{array}{c} 1\\12\end{array} \right]\right).
  \end{align*}

  \item Now suppose that the compartment transition is a merger, or coagulation.  Now we must select two compartments at random and combine their contents.  We will always choose that the combined contents of the compartments will be placed within the compartment with the lower index and will delete the compartment with the higher index.    Thus, assuming we choose the compartments indexed  1 and 2, we then merge the first and second compartments and place their contents into compartment 1 (since it has the smaller index of the two chosen) and then delete the second compartment. The resulting state is
  \begin{align*}
  \left( \left[ \begin{array}{c} 2\\3\end{array} \right], \left[ \begin{array}{c} 2\\2\end{array} \right], \left[ \begin{array}{c} 1\\12\end{array} \right]\right).
  \end{align*}

  \item Finally, we suppose that the compartment transition is a fragmentation.  The procedure will be as follows.  We will first choose the index of the compartment that fragments, we then create two new compartments and will then split the contents between these new compartments (with each particular molecule choosing between the new compartments with equal probability).  The originally chosen compartment will be deleted and the two new compartments will be placed at the end of the vector of states.  
  
  For example, suppose we choose compartment 3 for fragmentation (which occurs with probability $1/4$).   We then split the contents of the original third compartment (four molecules total, $2$ of $A$ and $2$ of $B$) uniformly at random between the two new compartments. Suppose for concreteness that we split as $\left[ \begin{array}{c} 1\\2\end{array}\right]$ and $\left[ \begin{array}{c} 1\\0\end{array}\right]$. Then, after deleting the 3rd compartment and adding these two onto the end we have a new state for the full model of
   \begin{align*}
  \left( \left[ \begin{array}{c} 2\\2\end{array} \right], \left[ \begin{array}{c} 0\\1\end{array} \right], \left[ \begin{array}{c} 1\\12\end{array} \right], \left[ \begin{array}{c} 1\\2\end{array}\right], \left[ \begin{array}{c} 1\\0\end{array}\right]\right).
  \end{align*}
  \end{itemize}

  It is also possible that the transition at time $T$ was due to a reaction taking place within one of the compartments. For example, if the reaction $A+2B\to A$ happens inside the fourth compartment, then the state of the whole system, $\Fsim$, will become
  \begin{align*}
  \left( \left[ \begin{array}{c} 2\\2\end{array} \right], \left[ \begin{array}{c} 0\\1\end{array} \right], \left[ \begin{array}{c} 2\\2\end{array} \right], \left[ \begin{array}{c} 1\\10\end{array} \right]\right).
  \end{align*}
  \hfill $\triangle$
\end{example}

Now we give the formal mathematical description of $\Fsim$. First, let $\{M_C(t)\}_{t\ge0}$ be the Markov chain associated to the compartment network $\compart_\Rate$. Then $M_C(t)$ will be the number of compartments at time $t$. Let $\{T_i\}_{i=0}^\infty$ be the jump times for this Markov chain, where $T_0=0$. For any $i\ge 0$ and any $j=1,\dots,M_C(T_i)$, let $\{X^i_j(t)\}_{t\in [T_i,T_{i+1}]}$ be realizations of the Markov chain associated to $\chem_\Rate$ with initial distributions (at time $T_i$) specified below. Suppose that for any $i_1,i_2$ and $j_1,j_2$ with either $i_1\ne i_2$ or $j_1\ne j_2$, the chains $X^{i_1}_{j_1}$ and $X^{i_2}_{j_2}$ are independent conditional on their initial conditions, and suppose that the initial distributions are chosen in the following manner (which are just formal characterizations of the details provided in the example above):
\begin{itemize}
    \item If the compartment transition at time $T_{i+1}$ was an inflow event ($0\to C$), then let $X^{i+1}_j(T_{i+1})=X^i_j(T_{i+1})$ for $j=1,\dots,M_C(T_i)$, and for $j=M_C(T_{i+1})=M_C(T_i)+1$ let $X^{i+1}_j(T_{i+1})$ be distributed according to $\mu$, independently of everything in the past.
    
    \item If the compartment transition at time $T_{i+1}$ was an exit event $(C\to 0$), then let $J_i$ be chosen uniformly at random from $\{1,\cdots,M_C(T_i)\}$, independently of everything in the past. Let $X^{i+1}_j(T_{i+1})=X^i_j(T_{i+1})$ for $j<J_i$, and let $X^{i+1}_j(T_{i+1})=X^i_{j+1}(T_{i+1})$ for $j\ge J_i$.
    
    \item If the compartment transition at time $T_{i+1}$ was a merger, or coagulation, 
    event ($2C\to C$), then let $J^1_i$ and $J^2_i$ be chosen uniformly at random from $\{1,\cdots,M_C(T_i)\}$ and $\{1,\cdots,M_C(T_i)\}\setminus\{J^1_i\}$, respectively, independent of everything in the past. Let $X^{i+1}_j(T_{i+1})=X^i_j(T_{i+1})$ for $j<\max\{J^1_i,J^2_i\}$ with $j\ne \min\{J^1_i,J^2_i\}$, let $X^{i+1}_j(T_{i+1})=X^i_{j+1}(T_{i+1})$ for $j\ge \max\{J^1_i,J^2_i\}$, and let $X^{i+1}_j(T_{i+1})=X^i_{J^1_i}(T_{i+1})+X^i_{J^2_i}(T_{i+1})$ for $j=\min\{J^1_i,J^2_i\}$.
    
    \item If the compartment transition at time $T_{i+1}$ was a fragmentation event ($C\to 2C$), then let $J_i$ be chosen uniformly at random from $\{1,\cdots,M_C(T_i)\}$, independently of everything in the past. Let $\{Z^i_k(x):x\in\ZZ^d,k=1,\dots,d\}$ be a collection of random variables, independent of each other and everything else, with $Z^i_k(x)\sim\mathrm{Binom}\left(0.5,x_k\right)$. Let $Z^i(x)$ denote the vector $\big(Z^i_1(x),\cdots,Z^i_d(x)\big)$. Let $X^{i+1}_j(T_{i+1})=X^i_j(T_{i+1})$ for $j<J_i$, let $X^{i+1}_j(T_{i+1})=X^i_{j+1}(T_{i+1})$ for $j=J_i,\dots,M_C(T_i)-1$, and for $j=M_C(T_i)$ let $X^{i+1}_j(T_{i+1})=Z^i(X^i_{J_i}(T_{i+1}))$ and $X^{i+1}_{j+1}(T_{i+1})=X^i_{J_i}(T_{i+1})-X^{i+1}_j(T_{i+1})$.
\end{itemize}

Let $\Fsim(t)$ be the vector $\left(X^i_1(t),X^i_2(t),\cdots,X^i_{M_C(t)}(t)\right)$, where $i$ is such that $T_i\le t<T_{i+1}$.

\begin{lemma}\label{lemma:sim-rep}
The process $\{\Fsim(t)\}_{t\ge0}$ is a continuous time Markov chain with state space $\bigcup_{m\ge0}\left(\ZZ_{\ge0}^d\right)^m$, the space of finite tuples of elements of $\ZZ_{\ge0}^d$.
\end{lemma}

\begin{proof}
To show that this is a Markov process we have to show that the holding times are exponential and the updates are independent of the holding times. To see that the holding times are exponential, notice that since $M_C$ is a Markov chain it has exponential holding times, and similarly for each $X^i_j$. But the holding times for these processes are independent, and the minimum of independent exponential random variables is itself exponential.

Furthermore, the minimum of a (finite) collection of independent exponential random variables is independent of the index at which the minimum occurs, so the updates are indeed independent of the holding times.

The fact that $\Fsim$ takes values in the space of finite tuples is equivalent to $M_C$ being finite for all time, which in turn is equivalent to the fact that $M_C$ is not explosive, regardless of the choice of rate constants in $\compart_\Rate$. This is a standard result in the theory of $1$-$d$ mass action stochastic reaction networks; see for instance \cite{Xu_Hansen_Wiuf_2022}.
\end{proof}

\subsubsection{An explicit construction of the simulation representation}\label{sec:construction}

We discuss one way of constructing the model described in Section \ref{sec:simulation}, in the spirit of the Kurtz representation \eqref{eq:RTC}. Here, by ``construction'' we mean an explicit detailing of the random processes and random variables needed to generate a single realization of the process.  The construction is of interest since it is amenable to analysis, coupling methods,  simulation methods, etc.   The construction will be used later in this paper to verify some behaviors of Example \ref{ex:chem positive recurrent}.

Let $\whole=(\chem_\Rate,\compart_\Rate,\mu)$ be as above. Suppose that $M_C(0)$ is the initial number of compartments in the system and further suppose that $M_C$ is given as the solution to 
\begin{align}
\begin{split}
M_C(t) &= M_C(0) + Y_I\left(  \kappa_I t\right)- Y_E\left(  \int_0^t \kappa_E M_C(s) ds\right)+ Y_F\left(  \int_0^t \kappa_F M_C(s) ds\right)\\
&\hspace{.5in}- Y_C\left(  \int_0^t \kappa_C \frac{M_C(s) (M_C(s) - 1)}{2}\, ds\right),\label{eq:compartments11}
\end{split}
\end{align}
where $Y_I, Y_E, Y_F$, and $Y_C$ are independent unit-rate Poisson processes. Then $M_C$ is the Markov chain on $\ZZ_{\ge 0}$ associated to $\compart_\Rate$, so that $M_C(t)$ gives the number of compartments at any time $t \ge 0$.

The jump times of the counting processes $R_I(t) = Y_I\left(  \kappa_I t\right)$, $R_E(t) = Y_E\left(  \int_0^t \kappa_E M_C(s) ds\right)$, $R_F(t) = Y_F\left(  \int_0^t \kappa_F M_C(s) ds\right)$, and $R_C(t) = Y_C\left(  \int_0^t \kappa_C \frac{M_C(s) (M_C(s) - 1)}{2}\, ds\right)$ determine when the RNIC model transitions due to changes in the count of the compartments.  To each such transition we will also require a collection of  random variables needed to carry out the updates in the RNIC model.  We detail these random variables below. Before proceeding, we remind the reader that a uniform random variable can be used to generate a sample from  a given distribution  in a number of ways.  For some standard transformation methods see, for example, \cite[Section 5.2]{anderson2017introduction}.

In the description below all random variables are independent of each other and of the Poisson processes $Y_I, Y_E, Y_F,Y_C$.  We require:
\begin{itemize}
    \item  A collection of independent uniform random variables $\{u_i^I\}$, $i = 1,2,\dots$.  When $R_I(T) - R_I(T-) = 1$, the random variable $u_{R_I(T)}^I$ is used to generate a sample from $\mu$.
    
    \item A collection of independent uniform random variables $\{u_i^E\}$, $i = 1,2, \dots$. When $R_E(T) - R_E(T-) = 1$,  the random variable $u_{R_E(T)}^E$ is used to determine which compartment exits at that time.
    
     \item Two collections of independent uniform random variables: (i) $\{u_i^{F}\}$,  $i = 1,2, \dots$, and (ii) an array $\{\hat u_{i,j}^F\}$, $i,j\in \{1,2,\dots\}$. When $R_F(T) - R_F(T-) = 1$,  the random variable $u_{R_F(T)}^F$ is used to determine which  compartment fragments.   We then utilize the finite collection $\{\hat u^F_{R_F(T),j}\}$, $j = 1,\dots, M$, where $M$ is the total number of molecules in the chosen compartment, to divide the different molecules between the two new cells.
    \item A collection of independent uniform random variables $\{u_i^C\}$, $i = 1,2, \dots$. When $R_C(T) - R_C(T-) = 1$,  the random variable $u_{R_C(T)}^C$ is used to determine which two compartments are chosen to merge.  
\end{itemize}
Note that the collections detailed above are chosen before a realization is generated.  Said differently, the realization of the RNIC model is a function of these independent random variables.

All that remains is to give the timing of the different chemical reactions.  One method is the following.  Let $\{Y_{r}\}_{r\in \Reac}$ be a collection of independent (of each other, and all other random objects so far), unit-rate Poisson processes, one for each possible reaction in $\chem_\Rate$.  Moreover, for each $r \in \Reac$, let $\{u^r_i\}$, $i=1,2,\dots$ be a collection of independent uniform random variables.  
Then, for $r \in \Reac$, we let 
\[
R_r(t) = Y_r \left( \mathop{\sum_{i\ge0}}_{T_i\le t}\int_{T_i}^{T_{i+1}\MIN t}\sum_{j=1}^{M_C(T_i)} \lambda_r(X^i_j(s)) ds\right),
\]
where the $T_i$ are the jump times of the process $M_C$, $X_j^i(s)$ is the state of the process in compartment $j$ at time $s$, and $\lambda_r$ is given as in \eqref{eq:lambda}.
Then $R_r$ is the counting process that jumps by $+1$ when the $r$th reaction takes place in \emph{some} compartment. 
 When $R_r(T) - R_r(T-) = 1$, meaning a reaction has taken place somewhere, we use $u^r_{R_r(T)}$ to determine the compartment within which the reaction took place.  In particular, the probability that it took place in compartment $k$ is simply
 \[
    \frac{\lambda_r(X_k^i(T-))}{\sum_{j=1}^{M_C(T_i)} \lambda_r(X_j^i(T-))}.
\]

\subsubsection{A coarse-grained representation}
\label{sec:coarse-grained}

While the description (and construction) above is often convenient for the sake of analysis and simulation, it is sometimes not the most natural way to think about these models. For example, suppose we have a model with a single species, denoted $S$, and for which there  are two compartments at time $t$, so that $M_C(t) = 2$. It is reasonable to think that we would not care to distinguish the situation in which there are $6$ molecules of species $S$ in the first compartment and $2$ in the second, which is the state $(6,2)$, versus the situation of $2$ molecules of $S$ in the first and $6$ in the second, which is the state $(2,6)$. In this situation, we would simply care that we have one compartment with two $S$ molecules, another with six, and there are no other compartments. 

To handle this, we  consider a function $n: \ZZ_{\ge 0} \to \ZZ_{\ge 0}$ in which $n_x:=n(x)$ gives the number of compartments present with precisely $x$ molecules of $S$ (hence the notation that ``$n$'' gives the \underline{\textit{n}}umber of compartments with different counts).  In this case,  the state of the example system described above would simply be the function with 
\[
    n_x = \begin{cases}
    1, & \text{ if } x=2\\
    1, & \text{ if } x = 6\\
    0, & \text{else}.
    \end{cases}
\]
Note that in this one-dimensional case we can also think of $n$ as an ``infinite vector.''  For example, in our example above we would have
\[
    n = (0,0,1,0,0,0,1,0,0,\dots),
\]
with only zeros continuing on.

For another example, we could consider the case discussed in Example \ref{ex:simulation representation}, where there are two species $A$ and $B$ and the state for the simulation representation was
\begin{align*}
  \left( \left[ \begin{array}{c} 2\\2\end{array} \right], \left[ \begin{array}{c} 0\\1\end{array} \right], \left[ \begin{array}{c} 2\\2\end{array} \right], \left[ \begin{array}{c} 1\\12\end{array} \right]\right).
\end{align*}
In this case, the state could naturally be described by the function
\[
    n_x = \begin{cases}
    2, & \text{ if } x = \left[ \begin{array}{c}
    2\\ 2 \end{array}\right]\vspace{.1in}\\ 
    1, & \text{ if } x= \left[ \begin{array}{c}
    0\\ 1 \end{array}\right]\vspace{.1in}\\
    1, & \text{ if } x= \left[ \begin{array}{c}
    1\\ 12 \end{array}\right]\\
    0, & \text{else}.
    \end{cases}
\]
Note that in this example, it is not natural to view $n$ as an ``infinite vector.''  Instead, it would be natural to view it as an ``infinite array'' with a two in the $(2,2)$ component, ones in the $(0,1)$ and $(1,12)$ components, and zeros elsewhere.

Thus, we may take the following approach, as done in \cite{Duso_Zechner_2020}.  
The state space of the coarse-grained model will be
\begin{align}\label{eq:coarsgrainedstatespace}
\begin{split}
\CoarseStateSpace:={}& \{ \text{functions } n: \ZZ^d_{\ge 0} \to \ZZ_{\ge 0} \text{ with compact support}\}\\
={}&\{ \text{functions } n: \ZZ^d_{\ge 0} \to \ZZ_{\ge 0} \text{ with finite support}\}\\
={}&\{ \text{functions } n: \ZZ^d_{\ge 0} \to \ZZ_{\ge 0} \text{ with finite $\ell^1$ norm}\},
\end{split}
\end{align}
where we observe that all three sets are the same. 
Given $n:\ZZ_{\ge0}^d\to\ZZ_{\ge0}$, we write $n=(n_x)_{x\in\ZZ^d_{\ge0}}$. For each possible state $x \in \ZZ^d_{\ge 0}$ of the chemistry, $n_x\in \ZZ_{\ge 0}$ represents the number of compartments whose chemistry has that particular state. Given Markov chains $M_C$ and $X^i_j$ as defined in Section \ref{sec:simulation}, let $N$ be the process where $N_x(t)$ is the number of compartments in state $x\in \ZZ^d_{\ge 0}$ at time $t \ge 0$:
\[
    N_x(t)=\sum_{i=0}^\infty\II\{t\in[T_i,T_{i+1})\}\sum_{j=1}^{M_C(T_i)}\II\{X^i_j(t)=x\}.
\]

Note that the total number of compartments at time $t\ge 0$ can be recovered from $N(t)$ via
\[
    M_C(t)= \norm{N(t)}_{\ell^1} := \sum_{x\in \ZZ^d_{\ge 0}} N_x(t).
\]
Note also that the process $N$ transitions iff $\Fsim$ does. This fact is important enough that we state it as a lemma:
\begin{lemma}\label{lemma:identical transition times}
Let $\Fsim$ and $N$ be as above. Then $N$ undergoes a transition at time $t$ iff $\Fsim$ does.
\end{lemma}

\begin{proof}
On the one hand, $N$ is defined as a function of $\Fsim$ and so $N$ cannot transition if $\Fsim$ does not. On the other hand, all possible transitions of $\Fsim$ cause a change in $N$: If $\Fsim$ transitions because $M_C$ does, then $\norm{N}_{\ell^1}=M_C$ changes, whereas if $\Fsim$ changes otherwise then the contents of some single compartment updated, which changes $N$.
\end{proof}

For the lemma below, we recall that for $x \in \ZZ^d_{\ge 0}$ we define $e_x$ to be the function taking the value of one at $x$ and zero otherwise.
\begin{lemma}
Let $N(t)$ be as defined above. Then $\{N(t)\}_{t\ge0}$ is a Markov chain taking values in $\CoarseStateSpace$, defined in \eqref{eq:coarsgrainedstatespace}. Moreover, for $n\in\CoarseStateSpace$, the transitions rates are as follows: 
\begin{center}
\begin{tabular}{lclcl}
Transition type &&   && Rate\\\hline
Compartment inflow && $n\mapsto n+e_x$ && $\kappa_I\mu(x)$\\[1ex]
Compartment exit && $n\mapsto n-e_x$ && $\kappa_En_x$\\[1ex]
Compartment coagulation, $x\ne y$ && $n\mapsto n+e_{x+y}-e_x-e_y$ && $\kappa_C n_xn_y$\\[1ex]
Compartment coagulation && $n\mapsto n+e_{2x}-2e_x$ && $\kappa_C \displaystyle\binom{n_x}{2}$\\[1ex]
Compartment fragmentation  && $n\mapsto n-e_{x+y}+e_x+e_y$ && $\kappa_F n_{x+y}\varphi(x+y,x)$\\
($x=y$ allowed here)\\[1ex]
Internal reaction $r \in \Reac$ && $n\mapsto n-e_x+e_{x+\nu_r'-\nu_r}$ && $n_x \kappa_r \displaystyle\binom{x}{\nu_r}$
\end{tabular}
\end{center}
where
\[
\varphi(z,x):=\prod_{k=1}^d\binom{z_k}{x_k}2^{-z_k}
\]
so that the distribution of the resulting compartments after a fragmentation is independently binomial in each species. Note that each row mentioning $x$ or $y$ corresponds to an infinite family of transitions and in the last row $r\in \Reac$ also ranges over all reactions of the reaction network $\chem$.
\end{lemma}
\begin{proof}
The fact that $N$ has finite support follows from the fact that $\Fsim$ is always a finite tuple, proved in Lemma \ref{lemma:sim-rep}.

The fact that $N$ is Markovian with the rates given follows from consideration of the infinitesimal behavior of $\Fsim$. For example, for $x\ne y \in \ZZ^d_{\ge 0}$, 
\begin{align*}
    \PP(N(t+h) = n+e_{x+y}-e_x - e_y | N(t) = n) &= \kappa_C n_x n_y h + o(h), \text{ as $h\to 0$},
\end{align*}
since, to leading order, the probability that some compartment in state $x$ merges with a compartment in state $y$ in the time interval $[t,t+h)$ is $\kappa_C n_x n_yh$.  The other rows of the table follow similarly.
\end{proof}

\begin{example}\label{Ex: Intro compartment model}
Consider the following  possible compartment model:
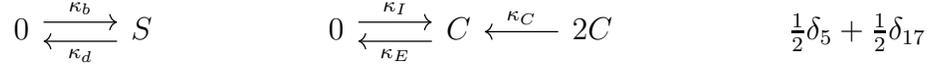
\begin{equation*}
\begin{tikzcd}
    0\arrow[yshift=.7ex]{r}{\kappa_b}& S\arrow[yshift=-.7ex]{l}{\kappa_d} &&
    0\arrow[yshift=.7ex]{r}{\kappa_I}& C\arrow[yshift=-.7ex]{l}{\kappa_E}&
    2C\arrow[swap]{l}{\kappa_C}&&
    \frac12\delta_5+\frac12\delta_{17}
\end{tikzcd}
\end{equation*}
Here we are keeping track of some chemical $S$ which forms with rate $\kappa_b$ and degrades with rate $\kappa_d$. Compartments are allowed to enter with rate $\kappa_I$, and new compartments that enter this way have either $5$ or $17$ molecules of $S$, each with probability $1/2$. Compartments can also exit with rate constant $\kappa_E$, and merge (or coagulate) with rate constant $\kappa_C$. Since there is only one species, the state space for the chemistry is $\ZZ_{\ge0}^1=\ZZ_{\ge0}$. As we detail below, we will be assuming mass-action kinetics; in this case that means when the model is in state $n\in\CoarseStateSpace$ the transition rates are given by
\begin{center}
\begin{tabular}{lclcl}
Transition type &&   && Rate\\\hline
Compartment inflow && $n\mapsto n+e_5$ && $\kappa_I/2$\\[1ex]
Compartment inflow && $n\mapsto n+e_{17}$ && $\kappa_I/2$\\[1ex]
Compartment exit && $n\mapsto n-e_x$ && $\kappa_En_x$\\[1ex]
Compartment coagulation ($x\ne y$) && $n\mapsto n+e_{x+y}-e_x-e_y$ && $\kappa_C n_xn_y$\\[1ex]
Compartment coagulation && $n\mapsto n+e_{2x}-2e_x$ && $\kappa_C \displaystyle \binom{n_x}{2}$\\[1ex]
$S$ birth && $n\mapsto n-e_x+e_{x+1}$ && $\kappa_b n_x$\\[1ex]
$S$ death && $n\mapsto n-e_x+e_{x-1}$ && $\kappa_d n_x x$
\end{tabular}
\end{center}
As before, each row mentioning $x$ or $y$ corresponds to an infinite family of transitions, one for each $x\ne y\in\ZZ^d_{\ge 0}$, and as always $e_x$ is the unit vector in direction $x$. \hfill $\triangle$
\end{example}

\section{Non-Explosivity}
\label{sec:explosivity}
A Markov Chain is explosive if it can undergo infinitely many transitions in finite time. 
The formal definition is below \cite{Norris_1997}.

\begin{defn}[Explosivity]
Let $\{X(t)\}_{t\ge0}$ be a continuous-time Markov chain with countable state space $\mathbb S$. For each $m\in\ZZ_{\ge0}$, let $\tau_m$ be the time of the $m$-th transition of $X$ (formally, $\tau_0=0$ and $\tau_m=\inf\{t>\tau_{m-1}:X(t)\ne X(\tau_{m-1})\}$), and let $\tau_\infty =  \lim_{m\to \infty} \tau_m$. We say that $X$ \textit{explodes} if $\tau_\infty<\infty$. If there is some state $x\in \mathbb S$ such that with positive probability $X$ explodes when started in state $x$, we say that $X$ is \textit{explosive}.
\end{defn}

We will show that explosivity for the RNIC model $\whole=(\chem_\Rate,\compart_\Rate,\mu)$ is determined by explosivity for the internal reaction network $\chem_\Rate$.   But to even talk about explosivity for $\whole$ instead of just the Markov chains $\Fsim$ or $N$, we need the following simple proposition.

\begin{prop}
Suppose we have a RNIC $\whole=(\chem_\Rate,\compart_\Rate,\mu)$. Let $\Fsim$ and $N$ be the corresponding simulation and coarse-grained representations. Then $\Fsim$ is explosive iff $N$ is explosive.
\end{prop}
\begin{proof}
This is immediate from lemma \ref{lemma:identical transition times}, which says that $\Fsim$ and $N$ transition at the  same times.
\end{proof}

In light of the proposition, we will speak merely of $\whole=(\chem_\Rate,\compart_\Rate,\mu)$ being explosive, and check the explosivity of either $\Fsim$ or $N$ depending on convenience. As it turns out, it will be most convenient to check explosivity for $\Fsim$. (Indeed, the fact that explositivity is more easily checked for $\Fsim$ is one of the major reasons for introducing $\Fsim$ in the first place.) 

\begin{theorem}
Suppose we have a RNIC $\whole=(\chem_\Rate,\compart_\Rate,\mu)$. Then $\whole$ is explosive iff $\chem_\Rate$ is explosive.
\end{theorem}
\begin{proof}
First, suppose that $\chem_\Rate$ is explosive. As discussed above, we intend to show that $\Fsim$ is explosive. By assumption, there is some $x\in\ZZ^d$ such that when the Markov chain corresponding to $\chem_\Rate$ is started in state $x$ it explodes with positive probability. In particular, there is some finite (nonrandom) time $t$ so that the chemistry undergoes infinitely many transitions before time $t$ with positive probability. Start $\Fsim$ in the state with one compartment whose state is $x$. With positive probability, no compartment transitions happen before time $t$. But the compartment transition times are independent of what is happening inside them by construction, and the compartment evolves according to $\chem_\Rate$, so on the event that no compartment transition happens before time $t$ the compartment undergoes infinitely many transitions before time $t$ with positive probability. It follows that $\Fsim$ is explosive.

Conversely, suppose that $\chem_\Rate$ is not explosive. Note that $\compart$, the compartment network, is not explosive for any choice of rate constants (see e.g.~\cite{Xu_Hansen_Wiuf_2022}). So with probability one $\Fsim$ undergoes only finitely many compartment transitions in finite time. But between each pair of consecutive compartment transitions there are finitely many compartments each evolving according to $\chem_\Rate$, and by assumption each of these undergoes only finitely many reactions in finite time a.s.. It follows that $\Fsim$ undergoes only finitely many transitions total in finite time, and hence is not explosive.
\end{proof}

\section{Transience, recurrence, and positive recurrence}
\label{sec:recurrenc}

The following definitions are standard.  For example, see \cite{Norris_1997}.

\begin{defn}
Let $M$ be a Markov chain with countable state space $\mathbb S$, and for $x\in\mathbb S$ let $T_x=\inf\{t>0:M_t=x\text{ but }\exists s\in[0,t], M_s\ne x\}$ be the first time the process returns to $x$ (or just arrives at $x$, if the process does not start from $x$). If $\PP_x(T_x<\infty)=1$, we say that the state $x$ is \emph{recurrent}, and if $\EE_x(T_x)<\infty$ we say that the state $x$ is \emph{positive recurrent}. A state which is not recurrent is called \emph{transient}, and a recurrent state which is not positive recurrent is \emph{null recurrent}. If $\PP_x(T_y<\infty)>0$ we say that $y$ is \emph{reachable} from $x$. If every state $x\in \mathbb{S}$  is positive recurrent, null recurrent, or transient, we say $M$ is \emph{positive recurrent}, \emph{null recurrent}, or \emph{transient}, respectively.
\end{defn}

A standard fact about (positive) recurrence is that it is a class property:
\begin{prop}[Theorems 3.4.1(iv) and 3.5.3(i)$\iff$(ii) in \cite{Norris_1997}]
Suppose that $y$ is reachable from $x$ and $x$ is recurrent (resp.\ positive recurrent). Then $y$ is recurrent (resp.\ 
positive recurrent).
\end{prop}

In other words, if you can get between $x$ and $y$ with positive probability (in both directions), then $x$ and $y$ are either both transient, both null recurrent, or both positive recurrent. So for irreducible chains (ones where you can pass between any two points of the state space with positive probability), the chain $M$ is always positive recurrent, null recurrent, or transient.

Before proceeding with the theory, we summarize the results of this section with a table. The way to read Table \ref{tab:Results Summary} is as follows:
\begin{itemize}
    \item Suppose we have a RNIC $(\chem_\Rate,\compart_\Rate,\mu)$, and $N$ is the associated coarse-grained model.

    \item The top row indicates possible dynamics (transient, null recurrent, or positive recurrent) for $\chem_\Rate$, the chemical model, and the left column indicates possible dynamics for $\compart_\Rate$, the compartment model. Since the possible dynamics for $N$ will turn out to depend crucially on whether the compartments can exit ($\kappa_E>0$) or not ($\kappa_E=0$), the left column is further subdivided along these lines.
    
    \item Several cells are marked ``Impossible", because $\compart_\Rate$ cannot be null recurrent if $\kappa_E=0$.

    \item The numbers inside each cell refer to the relevant theorems, lemmas, or examples that demonstrates the result. 
\end{itemize}

\begin{table}[ht]
\centering
\begin{tabular}{cc|c||c|c|c}
 &\multicolumn{2}{c}{}&\multicolumn{2}{c}{Chemistry ($\chem_\Rate$)}\\
    \multicolumn{2}{c}{}&& Transient (Trans.) & Null Recurrent (NR) & Positive Recurrent (PR)\\\cline{2-6}\noalign{\vskip\arrayrulewidth \vskip\doublerulesep}\cline{2-6}
    &\multicolumn{2}{c||}{\multirow{2}{*}{Transient}}& \multicolumn{3}{c}{$N$ must be Transient}\\
    &\multicolumn{2}{c||}{}& \multicolumn{3}{c}{Remark \ref{Rmk:Recurrence}}\\\cline{2-6}
    \multirow{8}{*}{\rotatebox[origin=c]{90}{Compartments ($\compart_\Rate$)}}&\multicolumn{1}{c|}{\multirow{4}{*}{NR}} & \multirow{2}{*}{$\kappa_E=0$} & \multicolumn{3}{c}{Impossible}\\
    &&& \multicolumn{3}{c}{Lemma \ref{Lem: Compartment Dynamics} }\\\cline{3-6}
    && \multirow{2}{*}{$\kappa_E>0$}& \multicolumn{3}{c}{$N$ must be Null Recurrent}\\
    &&& \multicolumn{3}{c}{Theorem \ref{Thm:Recurrence}}\\\cline{2-6}
    &\multicolumn{1}{c|}{\multirow{6}{*}{PR}} & \multirow{4}{*}{$\kappa_E=0$} & $N$ can be Trans.& $N$ can be Trans. & $N$ can be Trans.\\
     &&& Ex \ref{ex:both trans} & Ex \ref{Ex:Null recurrent chem} & Ex \ref{ex:bad mu}, \ref{ex:conservation law}, \ref{ex:chem positive recurrent}\\
    &&& $N$ can be PR& $N$ can be PR & $N$ can be PR\\
    &&& Ex \ref{ex:chem trans N recurrent} & Ex \ref{ex:chem null rec, N pos rec} & Ex \ref{ex:conservation law}\\\cline{3-6}
    && \multirow{2}{*}{$\kappa_E>0$}& \multicolumn{3}{c}{$N$ must be Positive Recurrent}\\
    &&& \multicolumn{3}{c}{Theorem \ref{Thm:Recurrence}}
\end{tabular}
\caption{The possibly dynamics for $N$, classified in terms of the dynamics for $\compart_\Rate$ and $\chem_\Rate$. In the above ``NR'' and ``PR'' stand for ``null recurrent'' and ``positive recurrent'', respectively, whereas ``Trans.'' stands for ``transient.''
}\label{tab:Results Summary}
\end{table}

Note that in all cases where we give an example of a recurrent $N$, the example is actually positive recurrent. We suspect that null recurrent examples will also exist, but we felt it more interesting to cover the behavioral extremes.

Moving to our theory, we begin by considering the dynamics of the compartment model of section \ref{sec:compartment}, which takes the form of a relatively simple reaction network, namely,
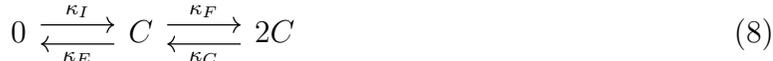
\begin{equation}\label{eq:compartment network}
\begin{tikzcd}
    0\arrow[yshift=.7ex]{r}{\kappa_I}& C\arrow[yshift=-.7ex]{l}{\kappa_E}\arrow[yshift=.7ex]{r}{\kappa_F}& 2C\arrow[yshift=-.7ex]{l}{\kappa_C}
\end{tikzcd}
\end{equation}
The  (positive) recurrence of this model  is already completely classified \cite{Xu_Hansen_Wiuf_2022}. We state this classification now as a lemma.

\begin{lemma}\label{Lem: Compartment Dynamics}
Consider the CRN in \eqref{eq:compartment network}.

\begin{itemize}
    \item Suppose  $\kappa_I=0$. Then $0$ is an absorbing state. If some other rate constant is non-zero then all other states are transient, whereas if all four rate constants are zero then all states are absorbing.

\item Suppose $\kappa_I>0$ and $\kappa_E > 0$. The irreducible state space is $\{0,1,2,\dots\}$ and:
\begin{itemize}
    \item If $\kappa_C>0$, then the chain is positive recurrent.
    \item If $\kappa_C=0$ but $\kappa_F<\kappa_E$, then the chain is  positive recurrent.
    \item If $\kappa_C=0$ and $\kappa_F>\kappa_E$, then the chain is transient.
    \item If $\kappa_C=0$ and $\kappa_F=\kappa_E$, then either $\kappa_I\le \kappa_E$ and the chain is null recurrent, or $\kappa_I>\kappa_E$ and the chain is transient.
\end{itemize}
\item Suppose $\kappa_I>0$ and $\kappa_E = 0$. Then all statements remain the same as in the case $\kappa_I>0$ and $\kappa_E > 0$ except the irreducible state space is now $\{1,2,\dots\}$ (and the state 0 is transient).
\end{itemize}
\end{lemma}

Now we begin with our positive results. The first fact is simple enough to be stated as a remark:

\begin{remark}\label{Rmk:Recurrence}
Notice that if $N$ is the course-grained representation for $\whole=(\chem_\Rate,\compart_\Rate,\mu)$ and $n$ is a (positive) recurrent state for $N$, then the number of compartments in $n$, $\norm n_{\ell^1}$, is a (positive) recurrent state for $\compart_\Rate$, since the return time to $\norm n_{\ell^1}$ is bounded by the return time to $n$.  
\end{remark}

Said succinctly, if $n$ is a positive recurrent state of the full model, then so is $\norm n_{\ell^1}$ for the compartment model. One might hope that the converse would be true, and it turns out under relatively mild assumptions it is:

\begin{theorem}\label{Thm:Recurrence}
Consider a non-explosive model $\whole=(\chem_\Rate,\compart_\Rate,\mu)$ where $\kappa_E>0$, and let $N$ be its course-grained representation. Then a state $n$ is (positive) recurrent for $N$ iff $n$ is reachable from the empty state $\vec 0$ for $N$ and the state $\norm n_{\ell^1}$ is (positive) recurrent for $\compart_\Rate$.
\begin{proof}
If $\kappa_I=0$ the conclusions of the theorem are clear, since by Lemma \ref{Lem: Compartment Dynamics} the state with no compartments is absorbing for both $N$ and $\compart_\Rate$ and all other states are transient. From here on we assume $\kappa_I>0$.

Let $M_C=\norm N_{\ell^1}$ be the number of compartments; recall that $M_C$ is a Markov chain which evolves according to $\compart_\Rate$. Suppose first that $n$ is recurrent for $N$. By Remark \ref{Rmk:Recurrence}, $\norm n_{\ell^1}$ is recurrent for $\compart_\Rate$. Since $\kappa_E>0$ and $\kappa_I>0$, by Lemma \ref{Lem: Compartment Dynamics} $\compart_\Rate$ is irreducible, so $\compart_\Rate$ eventually hits zero with probability one when started from $\norm n_{\ell^1}$. But when $M_C$ hits zero, $N=\vec 0$. Since $n$ is recurrent for $N$, it must be that $N$ eventually returns to state $n$ after hitting state $\vec 0$. This proves that $n$ is reachable from $\vec 0$ for $N$.

Now suppose that $n$ is reachable from $\vec 0$ and the state $\norm n_{\ell^1}$ is positive recurrent (resp. recurrent) for $\compart_\Rate$. Since $\compart_\Rate$ is irreducible as in the previous paragraph, it follows that zero is positive recurrent (resp. recurrent) for $\compart_\Rate$. But $N=\vec 0$ exactly when $M_C$ is $0$, so $\vec 0$ is positive recurrent (resp. recurrent) for $N$. But positive recurrence (resp. recurrence) is a class property and by assumption $n$ is reachable from $\vec 0$, so we conclude that $n$ is positive recurrent (resp. recurrent) for $N$, as desired.
\end{proof}
\end{theorem}

The same theorem holds, \textit{mutatis mutandis}, for $\Fsim$. The proof is the same so we omit it.
\begin{theorem}\label{Thm:another}
Consider a non-explosive model $\whole=(\chem_\Rate,\compart_\Rate,\mu)$ where $\kappa_E>0$, and let $\Fsim$ be its simulation representation. Then a state $(x_1,\dots,x_k)$ is (positive) recurrent for $\Fsim$ iff $(x_1,\cdots,x_k)$ is reachable from the empty vector $()$ for $\Fsim$ and the state $k$ is (positive) recurrent for $\compart_\Rate$.
\end{theorem}

\begin{remark}\label{rmk:PR_NMA}
Theorems \ref{Thm:Recurrence} and \ref{Thm:another} hold under more general assumptions.   Note that the key idea of both is that $0$ is (positive) recurrent for $\compart_\Rate$. Hence, one can generalize to the situation in which $\whole=(\chem_\Rate,\compart_\Rate,\mu)$ has non-mass action kinetics for either $\chem_\Rate$ or $\compart_\Rate$, so long as the system is non-explosive and $0$ is (positive) recurrent for $\compart_\Rate$.
\end{remark}

\subsection{Lyapunov Functions}

In what follows we will need to make use of the  theory of Lyapunov functions for Markov chains. This short section is devoted to introducing the extent of the theory  we will use.

The following theorem is well-known. In full generality, it is due to Meyn and Tweedie \cite{Meyn_Tweedie_1993}. The version below is a specialization to the countable state space case. For a proof of the version given below, see the later paper \cite{Anderson_Cappelletti_Kim_Nguyen_2020}.
\begin{theorem}\label{thm:lyapunov-recurrence}
    Let $X$ be a continuous-time Markov chain on a countable state space $\mathbb S$ with generator $\mathcal L$. Suppose there exists a finite set $K\subset\mathbb S$ and a positive function $V$ on $\mathbb S$ such that
    \[
    \mathcal LV(x)\le-1
    \]
    for all $x\in\mathbb S\setminus K$.  Suppose further that $V$ is ``norm-like,'' in the sense that $\{x\in\mathbb S:V(x)<B\}$ is finite for every $B>0$. Then each state in a closed, irreducible component of $\mathbb S$ is positive recurrent. Moreover, if $\tau_{x_0}$ is the time for the process to enter the union of the closed irreducible components given an initial condition $x_0$, then $\EE_{x_0}[\tau_{x_0}]<\infty$.
\end{theorem}

We will also need the following, which provides a method to check for transience.

\begin{theorem}\label{thm:lyapunov-transience}
    Let $X$ be a non-explosive continuous-time Markov chain on a countable discrete state space $\mathbb S$ with generator $\mathcal L$. Let $B\subset\mathbb S$, and let $\tau_B$ be the time for the process to enter $B$. Suppose there is some bounded function $V$ such that for all $x\in B^c$,
    \[
    \mathcal LV(x)\ge0.
    \]
    Then $\PP_{x_0}(\tau_B<\infty)<1$ for any $x_0$ such that
    \[
    \sup_{x\in B}V(x)<V(x_0).
    \]
\end{theorem}

For a version of the theorem above that applies in much greater generality, see Theorem 3.3(i) in \cite{Stramer_Tweedie_1994}. Our theorem is not an immediate corollary of theirs (they define restricted versions of the chain $X$ and state their theorem in terms of the generators of the restricted processes), so we will provide a proof of Theorem \ref{thm:lyapunov-transience} in the appendix.

\subsection{Instructive examples}
We now consider some examples. The first is an application of Theorem \ref{Thm:Recurrence}, and the rest show the various ways the conclusion of the theorem can fail if the hypothesis $\kappa_E>0$ is not satisfied. These examples also serve to illustrate various techniques that are useful for analysing recurrence and transience of RNIC models. In Example \ref{ex:chem trans N recurrent}, positive recurrence for the RNIC is shown via a Lyapunov function, applying Theorem \ref{thm:lyapunov-recurrence}. In Example \ref{Ex:Null recurrent chem}, transience for the RNIC is shown via a Lyapunov function, applying Theorem \ref{thm:lyapunov-transience}. And in Example \ref{ex:chem positive recurrent}, transience for the RNIC is shown with the help of the construction of $\Fsim$ given in section \ref{sec:construction}.

In the following, any rate constants not specified are assumed to be positive.

\begin{example}\label{ex:Apply Thm}
Consider the following RNIC.
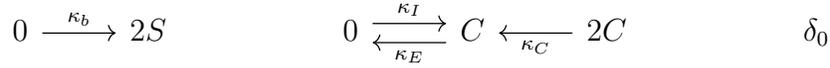
\begin{center}
\begin{tikzcd}
    0\arrow{r}{\kappa_b}& 2S & &
    0\arrow[yshift=.7ex]{r}{\kappa_I}& C\arrow[yshift=-.7ex]{l}{\kappa_E}& 2C\arrow{l}{\kappa_C} & &
    \delta_0
\end{tikzcd}
\end{center}
where $\delta_0$ is the point mass at zero (so each compartment enters empty). Even though $\chem_\Rate$ is transient, by Theorem \ref{Thm:Recurrence} the empty state is positive recurrent for $N$. Any state where every compartment has an even number of $S$ molecules is reachable from the empty state, hence positive recurrent. Any state where any compartment has an odd number of $S$ molecules is not reachable from the empty state, hence transient.\hfill $\triangle$
\end{example}

In all of the remaining examples in this section, we have $\kappa_E=0$ and hence the state $0$ will be transient for $\compart_\Rate$. Hence, when discussing the properties of the model we restrict ourselves to the state space $\CoarseStateSpace\setminus\{0\}$ that does not include the state with zero compartments.

The case where $\kappa_E=0$ is more complicated than the $\kappa_E\ne0$ case. For one thing, it is no longer enough just to look at $\compart_\Rate$ to decide if all states are transient. Indeed, if Example \ref{ex:Apply Thm} is modified so that $\kappa_E=0$ then every state becomes transient, despite the fact that all states are positive recurrent for the compartment network $\compart_\Rate$:

\begin{example}\label{ex:both trans}
Consider the model $\whole=(\chem_\Rate,\compart_\Rate,\mu)$ described by
\begin{equation}
\label{eq:both trans}
\begin{tikzcd}
    0\arrow{r}{\kappa_b}& 2S & &
    0\arrow{r}{\kappa_I}& C & 2C\arrow[swap]{l}{\kappa_C} & &
    \delta_0
\end{tikzcd}
\end{equation}
where $\delta_0$ is again the point mass at zero.

We reiterate that this is exactly the same as the previous example but with $\kappa_E$ set to zero. However, that is enough to make every state transient for $\whole$:

\begin{prop}
In the RNIC model \eqref{eq:both trans}, $\chem_\Rate$ is transient, $\compart_\Rate$ is positive recurrent on the irreducible state space $\{1,2,\dots\}$, and $N$ (the coarse-grained model corresponding to $\whole$) is transient.
\end{prop}
\begin{proof}
Except for the zero-compartment state (which cannot be returned to), all states are positive recurrent for $\compart_\Rate$ by Lemma \ref{Lem: Compartment Dynamics}. However,  the total number of $S$ molecules across all compartments can never shrink, and grows with some positive rate (at least $\kappa_b$, and larger if there are more compartments), so all states are transient for $N$.
\end{proof}
Thus we see that,  in this example, the long-term behavior of $\compart_\Rate$ and the course-grained model $N$ are different.
\hfill $\triangle$
\end{example}

The above example shows that when $\kappa_E=0$ and $\chem_\Rate$ is transient, $\whole$ may be transient even if $\compart_\Rate$ is not. However, this need not always be the case. Below we have an example that demonstrates that, when $\kappa_E=0$ and $\chem_\Rate$ is transient, it is still possible for $\whole$ to be positive recurrent.

\begin{example}\label{ex:chem trans N recurrent}
Consider the model $\whole=(\chem_\Rate,\compart_\Rate,\mu)$ described by
\begin{equation}\label{eq:chem trans N recurrent}
\begin{tikzcd}
    2A+B\arrow{r}{1}&B\arrow{r}{1}&0\arrow{r}{1}& A &
    0\arrow{r}{1}& C & 2C\arrow[swap]{l}{6} &
    \delta_{(0,1)}(a,b)
\end{tikzcd}
\end{equation}
where $\delta_{(0,1)}$ is a point mass with zero $A$ molecules and one $B$ molecule.  We will show that the chemical model $\chem_\Rate$ is transient but that the course-grained model, $N$, is positive recurrent.  Intuitively, this can be understood in the following manner: $B$ should be thought of as an enzyme that degrades the substrate $A$.  Without the compartment model, the enzyme would simply disappear over time, and then the substrate would grow without bound (from the reaction $0 \to A$).  However, each compartment brings in a new enzyme allowing for the further degradation of $A$.

\begin{prop}
In the RNIC model \eqref{eq:chem trans N recurrent}, $\chem_\Rate$ is transient, $\compart_\Rate$ is positive recurrent on the irreducible state space $\{1,2,\dots\}$, and $N$ (the coarse-grained model corresponding to $\whole$) is positive recurrent.
\begin{proof}
$\compart_\Rate$ is positive recurrent by Lemma \ref{Lem: Compartment Dynamics}. $\chem_\Rate$ is transient by the discussion above.  

It just remains to check positive recurrence of $N$. For $n\in\CoarseStateSpace$, let $C(n)=\norm n_{\ell^1}=\sum_{a=0}^\infty\sum_{b=0}^\infty n_{(a,b)}$ denote the number of compartments, and let $A(n)=\sum_{a=0}^\infty\sum_{b=0}^\infty an_{(a,b)}$ and $B(n)=\sum_{a=0}^\infty\sum_{b=0}^\infty bn_{(a,b)}$ be the total number of $A$ and $B$ molecules, respectively, across all compartments. Define $V:\CoarseStateSpace\to[0,\infty)$ via
\begin{align*}
    V(n)=\begin{cases}
        A(n)+B(n)+5C(n)-1   &   B(n)\ne0\\
        A(n)+B(n)+5C(n)+7   &   B(n)=0.
    \end{cases}
\end{align*}
We claim that this is a Lyapunov function for $N$. An upper bound for $\mathcal LV(n)$, the generator applied to $V$ at $n$, is given by
\begin{align*}
    \mathcal LV(n)\le\begin{cases}
        -B(n)+7-15C(n)(C(n)-1) & B(n)\ge2\text{ and }C(n)\ge2\\
        14-15C(n)(C(n)-1) & B(n)=1\text{ and }C(n)\ge2\\
        -1-15C(n)(C(n)-1) & B(n)=0\\
        -A(n)(A(n)-1)B(n)-B(n)+7 & B(n)\ge2\text{ and }C(n)=1\\
        -A(n)(A(n)-1)+14 & B(n)=1\text{ and }C(n)=1
    \end{cases}
\end{align*}
Note that the first two rows are upper bounds and the last three rows are exact. Specifically, in the first two rows we neglected the contribution of the $2A+B\to B$ reaction --- unlike everything else it crucially depends on how the $A$ and $B$ molecules are distributed across the compartments.

We see that $\mathcal LV(n)\le-1$ for all $n$ outside a finite set of states---for instance, you could take the states where there is exactly one compartment and it has at most $7$ $B$ and at most $4$ $A$. So $V$ is indeed a Lyapunov function for $N$, and hence $N$ is positive recurrent by Theorem \ref{thm:lyapunov-recurrence}.
\end{proof}
\end{prop}
\hfill$\triangle$
\end{example}

In the previous example we saw that even when $\kappa_E=0$, positive recurrent compartments $\compart_\Rate$ can still tame transient chemistry $\chem_\Rate$. It should not be surprising, then, that positive recurrent compartments can tame null recurrent chemistry in the same manner. For the sake of filling in Table \ref{tab:Results Summary} completely, we present a modification of Example \ref{ex:chem trans N recurrent} where $\chem_\Rate$ is null recurrent instead of transient.

\begin{example}\label{ex:chem null rec, N pos rec}
Consider the model $\whole=(\chem_\Rate,\compart_\Rate,\mu)$ described by
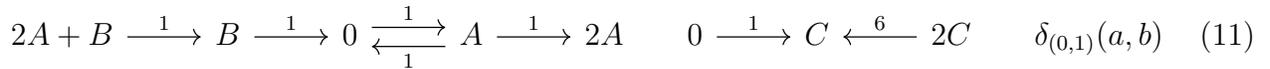
\begin{equation}\label{eq:chem null rec, N pos rec}
\begin{tikzcd}
   \hspace{-.2in} 2A+B\arrow{r}{1}&B\arrow{r}{1}&0\arrow[yshift=.7ex]{r}{1}& A\arrow[yshift=-.7ex]{l}{1}\arrow{r}{1}&2A &\hspace{-.2in}
    0\arrow{r}{1}& C & 2C\arrow[swap]{l}{6} &\hspace{-.2in}
    \delta_{(0,1)}(a,b)
\end{tikzcd}
\end{equation}
where $\delta_{(0,1)}$ is a point mass with zero $A$ molecules and one $B$ molecule.

The verification of this example is similar enough to that of Example \ref{ex:chem trans N recurrent} that we provide only a sketch.
\begin{prop}
In the RNIC model \eqref{eq:chem null rec, N pos rec}, $\chem_\Rate$ is null recurrent on the irreducible state space $\{0,1,2,\dots\}\times \{0\}$, $\compart_\Rate$ is positive recurrent on the irreducible state space $\{1,2,\dots\}$, and $N$ (the coarse-grained model corresponding to $\whole$) is positive recurrent.
\begin{proof}[Proof Sketch]
    Similarly to Example \ref{ex:chem trans N recurrent}, $\compart_\Rate$ is positive recurrent and $\chem_\Rate$ is eventually reduces (after all the $B$ molecules degrade) to the network
    
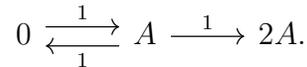
\begin{center}
    \begin{tikzcd}
        0\arrow[yshift=.7ex]{r}{1}& A\arrow[yshift=-.7ex]{l}{1}\arrow{r}{1}&2A.
    \end{tikzcd}
    \end{center}
    This model is null recurrent by Lemma \ref{Lem: Compartment Dynamics}.

    As for $N$, let $V$ be the very same Lyapunov function used to prove positive recurrence in Example \ref{ex:chem trans N recurrent}. The only difference between this example and that one is the addition of the reactions $A\to0$ and $A\to 2A$. But notice that the contribution of $A\to0$ in $\mathcal LV(n)$ is $-A(n)$, and the contribution of $A\to2A$ is $A(n)$. These are equal and opposite, so $\mathcal LV(n)$ is exactly the same in this example and Example \ref{ex:chem trans N recurrent}. Thus the remainder of the proof is identical.
\end{proof}
\end{prop}
\hfill$\triangle$
\end{example}

Examples \ref{ex:both trans} and \ref{ex:chem trans N recurrent} showed that $\whole=(\chem_\Rate,\compart_\Rate,\mu)$ can be either positive recurrent or transient when $\kappa_E=0$ and $\chem_\Rate$ is transient. The next few examples are dedicated to showing the same when $\chem_\Rate$ is recurrent. First, if new compartments enter with a huge number of molecules, it can overwhelm otherwise positive recurrent chemistry:

\begin{example}\label{ex:bad mu}
Consider the RNIC model $\whole=(\chem_\Rate,\compart_\Rate,\mu)$ described by
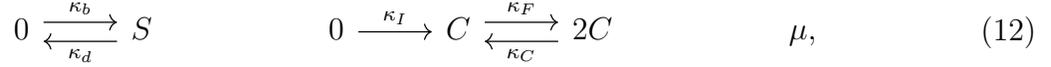
\begin{equation}\label{eq:113241234}
\begin{tikzcd}
    0\arrow[yshift=.7ex]{r}{\kappa_b}& S\arrow[yshift=-.7ex]{l}{\kappa_d}&& 
    0\arrow{r}{\kappa_I}& C\arrow[yshift=.7ex]{r}{\kappa_F}& 2C\arrow[yshift=-.7ex]{l}{\kappa_C} &&
    \mu,
\end{tikzcd}
\end{equation}
where $\mu$ is not yet specified.
\end{example}

\begin{prop}
Let $N$ be the coarse-grained model associated with the RNIC model \eqref{eq:113241234}. For any choice of non-negative rate constants such that $\kappa_I>0$, there is a distribution $\mu$ on the non-negative integers such that $N$ is transient.
\begin{proof}
We will show that in the case $\kappa_b = 0$, $\mu$ can be chosen so that the total number of $S$ molecules is itself a transient Markov chain. The case of $\kappa_b>0$ then immediately follows by a coupling argument. That portion of the proof is straightforward and is omitted. 

Let $M(t)$ denote the number of $S$ molecules across all compartments at time $t$. Under the assumption that $\kappa_b=0$, $M$ is a Markov chain which transitions from state $m\in\NN$ to state $m-1$ with rate $\kappa_d m$ and to state $m+j$ with rate $\kappa_I\mu(j)$.

Our plan is the following: we will recursively define an increasing sequence of integers $m_k$ for $k=1,2,3,\dots$, and define $\mu(m_k)=2^{-k}$ and $\mu(j)=0$ otherwise. For $k=2,3,4,\dots$, we will let $A_k$ denote the event that the process $M$ reaches $m_{k-1}$ before it reaches (or exceeds) $m_{k+1}$. It then suffices to show that $\sup_k \PP_{m_k}(A_k)<1/2$ to prove transience of $M$. 

Continuing, we begin by letting $m_1 = 0$.  Now suppose $m_1,\dots,m_{k-1}$ have been defined. We will show that for any $\varepsilon>0$ it is possible to pick $m_k$ so that $\PP_{m_k}(A_k)<\varepsilon$ regardless of the values chosen for $m_{k+1},m_{k+2},\dots$. To show this, we make the following observations.
\begin{enumerate}
    \item Since $M$ can only go down by one at a time, to get from $m_k$ to $m_{k-1}$ before hitting a state equal to or larger than $m_{k+1}$,  the process must visit every state $m_k,m_k-1,\cdots,m_{k-1}+1$ at least once.
    \item On the event $A_k$, during each visit to each of the states $m_{k-1}+1,\dots, m_k$ there was no transition of size $+m_{k+1}$ (for in that case the state of $M$ would would necessarily reach or exceed $m_{k+1}$).
\end{enumerate}
The probability of the process $M$ transitioning  up by $m_{k+1}$ while in state $m$ is $\frac{2^{-(k+1)}\kappa_I}{\kappa_I+\kappa_dm}$ because the total rate out of state $m$ is $\kappa_I+\kappa_dm$, and the rate of inflows of size $m_{k+1}$ in state $m$ is $\mu(m_{k+1})\kappa_I=2^{-k-1}\kappa_I$. Hence, combining the above observations we see
\begin{align*}
    \PP_{m_k}(A_k)
    &\le \prod_{m=m_{k-1}+1}^{m_k}\left(1-\frac{2^{-(k+1)}\kappa_I}{\kappa_I+\kappa_dm}\right)\\
    &\le \prod_{m=m_{k-1}+1}^{m_k}\exp\left(-\frac{2^{-(k+1)}\kappa_I}{\kappa_I+\kappa_dm}\right)
    =\exp\left(-2^{-(k+1)}\kappa_I\sum_{m=m_{k-1}+1}^{m_k}\frac{1}{\kappa_I+\kappa_dm}\right),
\end{align*}
where above we use the bound $1-x\le e^{-x}$.

If $m_{k-1}$ is fixed and we send $m_k\to\infty$ in the sum above, we get $\infty$ (it's a tail of a harmonic series). Therefore,  $\PP_{m_k}(A_k)$ can be made as small as we like by choosing $m_k$ big enough. We conclude that for appropriate choice of $m_k$, the process $M$ is transient, and hence so is $N$. 
\end{proof}

Hence, so long as $\kappa_E=0$, a distribution $\mu$ that is ``bad enough'' can cause the whole model to be transient even if the chemical model $\chem_\Rate$ is positive recurrent.
\hfill $\triangle$
\end{prop}

In the previous example, the distribution $\mu$ of incoming compartments was unbounded. As it turns out, $\whole=(\chem_\Rate,\compart_\Rate,\mu)$ can be transient even when $\chem_\Rate$ and $\compart_\Rate$ are positive recurrent and $\mu$ is bounded. The simplest, though not only, reason this can occur is the existence of some conservation law, as the next example demonstrates. Put simply, the total amount of species $A$ and $B$ is preserved by the chemistry, so any inflow of those species, no matter how small, will overwhelm it.

\begin{example}\label{ex:conservation law}
Consider the RNIC model $\whole=(\chem_\Rate,\compart_\Rate,\mu)$ described by
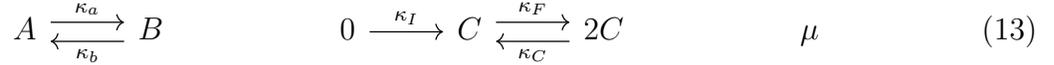
\begin{equation}\label{eq:conservation law}
\begin{tikzcd}
    A\arrow[yshift=.7ex]{r}{\kappa_a}& B\arrow[yshift=-.7ex]{l}{\kappa_b}&& 
    0\arrow{r}{\kappa_I}& C\arrow[yshift=.7ex]{r}{\kappa_F}& 2C\arrow[yshift=-.7ex]{l}{\kappa_C} && 
    \mu
\end{tikzcd}
\end{equation}
where $\mu$ is not yet specified.

\begin{prop}
Let $N$ be the coarse-grained model associated with the system $\whole$ from \eqref{eq:conservation law}.  If $\mu$ is any measure on $\ZZ_{\ge0}^2$ other than the trivial measure $\delta_{(0,0)}$, then $N$ is transient even though all states are positive recurrent for $\chem_\Rate$. On the other hand, if $\mu=\delta_{(0,0)}$ then $N$ is positive recurrent.
\begin{proof}
$\chem_\Rate$ is not irreducible, but when it is partitioned into closed irreducible communicating classing, all are finite, and hence all states are positive recurrent. As always when $\kappa_E=0$ but $\kappa_C>0$, the empty state is transient for $\compart_\Rate$ but all other states are positive recurrent.

For $n\in\CoarseStateSpace$, let $S(n)=\sum_{a=0}^\infty\sum_{b=0}^\infty (a+b)n_{(a,b)}$ denote the sum of the number of $A$ and $B$ molecules, combined across all compartments in $n$.

First suppose that $\mu\ne\delta_{(0,0)}$. Then $S(N(t))$ cannot shrink, and grows with positive probability every time a compartment enters. So $N$ is transient in this case.

Now suppose $\mu=\delta_{(0,0)}$. For $n\in\CoarseStateSpace$, let $C(n)=\norm n_{\ell^1}$ be the number of compartments in state $n$, and let $V(n)=2C(n)$. Then
\begin{align*}
    \mathcal LV(n)=2\kappa_I+2\kappa_FC(n)-\kappa_CC(n)(C(n)-1),
\end{align*}
where $\mathcal L$ is the generator of $N$. This is less than $-1$ outside a finite set because it is quadratic in $C(n)$ with negative leading term, provided we restrict the state space to $\{n\in\CoarseStateSpace: S(n)=S(N(0))\}$. So Theorem \ref{thm:lyapunov-recurrence} applies and $N$ is positive recurrent, as claimed.
\end{proof}
\end{prop}
\hfill $\triangle$
\end{example}

A natural question at this point is whether, if the behaviors in the last two examples are ruled out, $N$ can still be transient when $\chem_\Rate$ and $\compart_\Rate$ are both separately recurrent. Specifically, if $\chem_\Rate$ and $\compart_\Rate$ are both recurrent, there are no conservation laws, and the number of molecules that an incoming compartment can have is bounded, can $N$ be transient? The answer is yes, as the next example demonstrates.

\begin{example}\label{Ex:Null recurrent chem}
Consider the RNIC model $\whole=(\chem_\Rate,\compart_\Rate,\mu)$ described by
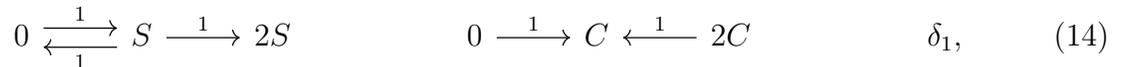
\begin{equation}\label{eq:Null recurrent chem}
\begin{tikzcd}
    0\arrow[yshift=.7ex]{r}{1}& S\arrow[yshift=-.7ex]{l}{1}\arrow{r}{1}& 2S& &
    0\arrow{r}{1}& C& 2C\arrow[swap]{l}{1} & &
    \delta_1,
\end{tikzcd}
\end{equation}
where $\delta_1$ is the point mass at one $S$.\hfill 

\begin{prop}
Let $N$ be the coarse-grained model associated to the network $\whole=(\chem_\Rate,\compart_\Rate,\mu)$ from \eqref{eq:Null recurrent chem}. Then $\chem_\Rate$ is recurrent with no conservation laws and the number of molecules in new compartments is bounded, however every state is transient for $N$.
\begin{proof}
$\chem_\Rate$ is (null) recurrent, and $\compart_\Rate$ is positive recurrent on the irreducible state space $\{1,2,\dots\}$, by Lemma \ref{Lem: Compartment Dynamics}.

It remains to show that every state is transient for $N$. As in all examples with $\kappa_E=0$, the state with zero compartments can never be returned to and we restrict the state space of the chain to $\CoarseStateSpace\setminus\{0\}$. With this assumption the state space is a closed irreducible set, so it suffices to pick one state and show that it is transient. We will show $e_0$ (the state with one empty compartment) is transient. Denoting a state of $N$ by $n$, let $C(n)=\sum_{x=0}^\infty n_x$ and $S(n)=\sum_{x=0}^\infty x\cdot n_x$ denote the total number of compartments and $S$ molecules, respectively. Define $V:\CoarseStateSpace\to[0,1]$ by
\[
V(n)=\frac{S(n)}{1+S(n)}.
\]
If $\mathcal L$ denotes the generator of $N$, notice that
\begin{align*}
    \mathcal LV(n)
    &=(C(n)+S(n)+1)\left(\frac{S(n)+1}{S(n)+2}-\frac{S(n)}{S(n)+1}\right)+S(n)\left(\frac{S(n)-1}{S(n)}-\frac{S(n)}{S(n)+1}\right)\\
    &=\frac{C(n)+S(n)+1}{(S(n)+2)(S(n)+1)}-\frac1{S(n)+1}\\
    &=\frac{C(n)-1}{(S(n)+2)(S(n)+1)}\\
    &\ge0
\end{align*}
for all $n\in\CoarseStateSpace\setminus\{0\}$. In particular, if $B=\{e_0\}$, we can apply Theorem \ref{thm:lyapunov-transience} to conclude that when $N$ is started from $e_0+e_1$ (the state with two compartments, one empty and the other with one $S$), then the probability of reaching $B$ is less than $1$. But when $N$ is started from $e_0$, it reaches $e_0+e_1$ with positive probability (the transition from $e_0$ to $e_0+e_1$ corresponds to an inflow event). Putting these together, when $N$ is started from $e_0$ it fails to return with positive probability, and hence $e_0$ is transient. As discussed, this is enough to conclude that all states are transient for $N$.
\end{proof}
\end{prop}
\hfill $\triangle$
\end{example}

In the previous example $\chem_\Rate$ was \emph{null} recurrent.  One may still be tempted to think that perhaps if it were \emph{positive} recurrent then the whole process must be.  The next example demonstrates that even this is not guaranteed.

\begin{example}\label{ex:chem positive recurrent}
Consider the compartment model described by
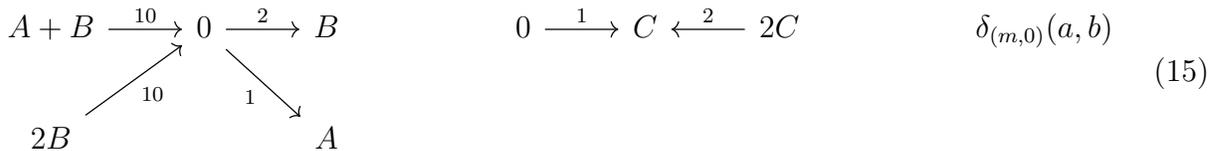
\begin{equation}\label{eq:chem positive recurrent}
\begin{tikzcd}
   A+B\arrow{r}{10}&0\arrow{r}{2}\arrow[swap]{dr}{1}& B&&
    0\arrow{r}{1}& C& 2C\arrow[swap]{l}{2} &&
    \delta_{(m,0)}(a,b)\\
    2B\arrow{ur}[swap]{10}&&A
\end{tikzcd}
\end{equation}
where $m$ is some non-negative integer and $\delta_{(m,0)}$ is the point mass at $m$ molecules of $A$ and zero of $B$. Let $\gamma>0$ denote the expected number of compartments in stationarity.

\begin{prop}
Let $\whole=(\chem_\Rate,\compart_\Rate,\mu)$ be the compartment model from \eqref{eq:chem positive recurrent}, and let $N$ be the associated coarse-grained model. Then $\chem_\Rate$ is positive recurrent, but $N$ is transient when $m > \gamma$. 
\begin{proof}
That $\chem_\Rate$ is positive recurrent is witnessed by the Lyapunov function
\begin{align*}
    V(a,b)=\begin{cases}
        3a+3        &   b=0\\
        3a+3b-2     &   b\ge1
    \end{cases}
\end{align*}
Indeed, if $\mathcal A$ denotes the generator of $\chem_\Rate$, then
\begin{align*}
\mathcal A V(a,b)&=\begin{cases}
    3(1)-2(2) & b=0\\
    3(1)+3(2)-1(10a) & b=1\\
    3(1)+3(2)-6(20a)-1(10) & b=2\\
    3(1)+3(2)-6(10ab)-6(5b(b-1)) & b\ge3
\end{cases}\\
&=\begin{cases}
    -1 & b=0\\
    9-10a & b=1\\
    -1-120a & b=2\\
    9-60ab-30b(b-1) & b\ge3
\end{cases}
\end{align*}
This is at most $-1$ away from $(0,1)$, so by Theorem \ref{thm:lyapunov-recurrence} $\chem_\Rate$ is positive recurrent.

Now regarding transience of $N$, let $\Fsim$ be the simulation representation of $\whole$, so that $N$ is a fuction of $\Fsim$. Let $X_A$ and $X_B$ denote the total number of $A$ and $B$ molecules, respectively, across all compartments in $N$ (equivalently, across all compartments in $\Fsim$). To show that $N$ is transient, we will show that $X_A(t)\to\infty$ a.s., as $t\to\infty$. To do this, we will make use of the construction of $\Fsim$ from section \ref{sec:construction}. Let $Y_I$ and $Y_C$ be as in that section, so that the process $M_C$ for the number of compartments is given by
\begin{align*}
M_C(t) &= M_C(0) + Y_I\left(t\right)- Y_C\left(  \int_0^t \frac{M_C(s) (M_C(s) - 1)}{2} ds\right).
\end{align*}
Similarly, for $r\in\{A+B\to0,0\to B,2B\to0,0\to A\}$ let $Y_r$ be as in section \ref{sec:construction}, and let $R_r$ be the associated counting process for the number of times reaction $r$ has occurred across all compartments, so that
\[
R_r(t) = Y_r \left( \mathop{\sum_{i\ge0}}_{T_i\le t}\int_{T_i}^{T_{i+1}\MIN t}\sum_{j=1}^{M_C(T_i)} \lambda_r(X^i_j(s)) ds\right),
\]
where the $T_i$ are the jump times of the process $M_C$, $X_j^i(s)$ is the state of the process in compartment $j$ at time $s$, and $\lambda_r$ is given as in \eqref{eq:lambda}. Then
\begin{align*}
X_A(t)
&=X_A(0)+R_{0\to A}(t)+mY_I\left(t\right)-R_{A+B\to0}(t)\\
&=X_A(0)+Y_{0\to A}\left(\int_0^t M_C(s) ds\right)+mY_I\left(t\right)-R_{A+B\to0}(t).
\end{align*}
Notice that in the last line above we were able to simplify the expression for $R_{0\to A}$ in terms of $Y_{0\to A}$ from the expression given above for $R_r$ in general. This was done by making use of the fact that the total rate of this reaction across all compartments, $\sum_j \lambda_{0\to A}(X_j^i(s))$, is exactly the total number of compartments $M_C(s)$. We cannot hope to do the same for $R_{A+B\to0}$ because the rate of that reaction depends on how the molecules are distributed across the compartments. However, notice that the total number of times the reaction $A+B\to0$ fires is at most the total number of $B$ molecules ever present in the system:
\begin{align*}
R_{A+B\to0}(t)
&\le X_B(0)+R_{0\to B}(t)\\
&=X_B(0)+Y_{0\to B}\left(2\int_0^t M_C(s) ds\right).
\end{align*}
Therefore,
\begin{align*}
X_A(t)
&\ge X_A(0)-X_B(0)+Y_{0\to A}\left(\int_0^t M_C(s) ds\right)+mY_I\left(t\right)-Y_{0\to B}\left(2\int_0^t M_C(s) ds\right).
\end{align*}
Recall that $\gamma$ denotes the expected number of $C$ in the CRN $\compart_\Rate$ at stationarity. By the CTMC ergodic theorem (see Theorem 45 in Chapter 4 of \cite{Serfozo_2009}), $\frac1t\int_0^t M_C(s)ds\to\gamma$ almost surely as $t\to\infty$. This will matter in its own right; it also follows that $\int_0^t M_C(s)ds\to\infty$ a.s.\ as $t\to\infty$. It is a standard fact about unit Poisson processes $Y$ that $Y(t)/t\to 1$ a.s.\ as $t\to\infty$. Composing this Poisson limit with the limit from the previous sentence, we get that
\begin{align*}
    \frac{Y_{0\to B}\left(2\int_0^tM_C(s)ds\right)}{2\int_0^tM_C(s)ds}\to 1
\end{align*}
a.s. as $t\to\infty$, and similarly for $Y_{0\to A}$. Putting this all together we have
\begin{align*}
     \lim_{t\to\infty}\frac{X_A(t)}t
     &\ge\lim_{t\to\infty}\Bigg[ \frac{Y_{0\to A}\left(\int_0^tM_C(s)ds\right)}{\int_0^tM_C(s)ds}\cdot\frac1t\int_0^tM_C(s)ds+m\frac{Y_I(t)}t\\
     &\hspace{1in}-\frac{Y_{0\to B}\left(2\int_0^tM_C(s)ds\right)}{2\int_0^tM_C(s)ds}\cdot\frac2t\int_0^tM_C(s)ds\Bigg]\\
     &=\gamma+m-2\gamma.
\end{align*}
almost surely. Therefore, as long as the integer $m$ is (strictly) larger than $\gamma$, $X_A(t)/t$ is converging almost surely to a positive number. In this case $X_A(t)\to\infty$ a.s.\ as $t\to\infty$, and hence $N$ is transient.
\end{proof}
\end{prop}
Note that the above example shows the potential usefulness of the RNIC representation provided in section \ref{sec:construction}.\hfill
$\triangle$
\end{example}

\section{Stationary Distribution in a Special Case}
\label{sec:SD}

In light of Theorem \ref{Thm:Recurrence}, whenever $\compart_\Rate$ is positive recurrent and $\kappa_E>0$, then $N$, the coarse-grained model associated to $\whole=(\chem_\Rate,\compart_\Rate,\mu)$, is positive recurrent for at least some states. In this case, the standard theory of Markov chains tells us that there is a stationary distribution supported on those states. Ideally, it would be possible to write down a formula for this stationary distribution in terms of information about the CRNs $\chem_\Rate$ and $\compart_\Rate$. Under the further assumption that $\kappa_C=0=\kappa_F$ (so that compartments are not interacting), we are able to do so.

\begin{theorem}\label{Thm:Stationary Distribution}
Consider a non-explosive model $\whole=(\chem_\Rate,\compart_\Rate,\mu)$ with $\kappa_F=\kappa_C=0$, and $\kappa_E>0$:
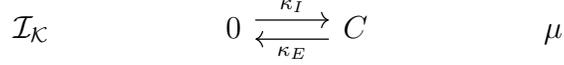
\begin{equation*}
\begin{tikzcd}
    \chem_\Rate &&
    0\arrow[yshift=.7ex]{r}{\kappa_I}& C\arrow[yshift=-.7ex]{l}{\kappa_E}
    &&
    \mu
\end{tikzcd}
\end{equation*}
Let $N$ be the coarse-grained model associated to $\whole$. For $x\in\ZZ_{\ge0}^d$ and $t\in[0,\infty)$, let $P_\mu(x,t)$ denote the probability that $\chem_\Rate$ is in state $x$ at time $t$ when started from time zero with initial distribution $\mu$. For $x\in\ZZ_{\ge0}^d$ define $\alpha(x)$ via
\begin{align*}
    \alpha(x)=\int_0^\infty P_\mu(x,t) \kappa_E e^{-\kappa_E t}\mathrm dt,
\end{align*}
and define a distribution $\pi$ on $\CoarseStateSpace$ via
\begin{align*}
\pi(n)=\left(\prod_{x\in\ZZ_{\ge0}^d}\frac{\alpha(x)^{n_x}}{n_x!}\right)\cdot\left[e^{-\kappa_I/\kappa_E}\cdot\left(\frac{\kappa_I}{\kappa_E}\right)^{\norm n_{\ell^1}}\right]
\end{align*}
Then $\pi$ is the unique stationary distribution for $N$. 

\begin{remark}\label{rmk:DR}
To apply Theorem \ref{Thm:Stationary Distribution}, one needs to know, not just the stationary distribution for the chemistry, but the distribution for all time. This restriction may seem daunting, and indeed for many models this distribution is not known. One class of models where it \textit{is} know are the DR models of \cite{Anderson_Schnoerr_Yuan_2020}. A second class of models are monomolecular reaction networks with arbitrary initial conditions --- see \cite{Jahnke_Huisinga_2007}. Note that \cite{Anderson_Schnoerr_Yuan_2020} allows for more general networks (all monomolecular networks satisfy the DR condition), but \cite{Jahnke_Huisinga_2007} allows for more general initial conditions (the DR paper requires Poisson initial conditions).
\end{remark}

\begin{proof}[Proof of Theorem \ref{Thm:Stationary Distribution}.]
Note that by Theorem \ref{Thm:Recurrence}, any state which is reachable from the zero state is positive recurrent, and all other states are transient. Furthermore, notice that $N$ is irreducible if restricted to the set of states which are reachable from the zero state, since zero is reachable from any state. Thus there is a unique stationary distribution. To prove that the $\pi$ given above is indeed this unique stationary distribution, it suffices to show that $\pi$ is a distribution and $\pi Q=0$, where $Q$ is the transition rate matrix for $N$. That $\pi$ is a distribution follows from the fact that $\alpha$ is a distribution, which we will check later in the proof. So fix $n\in\CoarseStateSpace$; we wish to show that $\sum_{n'\in\CoarseStateSpace}\pi(n') q(n',n)=0$.

Note that there are only three possible types of transitions: inflow of compartment, outflow of compartment, and transition of reaction network. Expanding the sum above into three terms, one for each of these types of transitions, the desired equality can be written 
\begin{align*}
    \sum_{x\in\ZZ_{\ge0}^d}\Bigg[\pi(n-e_x)&q(n-e_x,n)+\pi(n+e_x)q(n+e_x,n)\\\nonumber
    &+\sum_j \pi(n-e_x+e_{x-\nu'_j+\nu_j})q(n-e_x+e_{x-\nu'_j+\nu_j},n)\Bigg]\\\nonumber
    &=\pi(n)\sum_{x\in\ZZ_{\ge0}^d}\left(q(n,n+e_x)+q(n,n-e_x)+\sum_j q(n,n-e_x+e_{x+\nu'_j-\nu_j})\right)
\end{align*}
or 
\begin{align}\nonumber
    \sum_{x\in\ZZ_{\ge0}^d}\Bigg[\pi(n-e_x)&\kappa_I\mu(x)+\pi(n+e_x)\kappa_E (n_x+1)\\\label{eq:piQ=0}
    &+\sum_j \pi(n-e_x+e_{x-\nu'_j+\nu_j})(n_{x-\nu'_j+\nu_j}+1)\kappa_j\binom{x-\nu'_j+\nu_j}{\nu_j}\Bigg]\\\nonumber
    &=\pi(n)\sum_{x\in\ZZ_{\ge0}^d}\left(\kappa_I\mu(x)+\kappa_En_x+\sum_j n_x\kappa_j\binom{x}{\nu_j}\right)
\end{align}
To prove this equality, we will consider two cases. Suppose first that $n$ is such that $n_y>0$ for some $y\in\ZZ_{\ge0}^d$ with $\alpha(y)=0$, and fix such a $y$. Then $\alpha(y)$ participates in the product defining $\pi(n)$, and hence $\pi(n)=0$. Thus the right-hand side of \eqref{eq:piQ=0} is zero; we claim that the left-hand side is also zero. Specifically, we will argue for each $x$ and each $j$, each of the three terms in the sum is zero. So fix $x$ and $j$:
\begin{itemize}
    \item $\pi(n-e_x)\kappa_I\mu(x)$: Notice that if $x\ne y$ then $\pi(n-e_x)=0$ for the same reason that $\pi(n)=0$. If $x=y$ then $\mu(x)=0$, since if $\mu(y)>0$ it would be the case that $P_\mu(y,t)>0$ for all small enough $t$, and hence the integral defining $\alpha(y)$ would be positive.

    \item $\pi(n+e_x)\kappa_E (n_x+1)$: Regardless of $x$, $\pi(n+e_x)=0$ for the same reason that $\pi(n)=0$.

    \item $\pi(n-e_x+e_{x-\nu'_j+\nu_j})(n_{x-\nu'_j+\nu_j}+1)\kappa_j\binom{x-\nu'_j+\nu_j}{\nu_j}$: As before, if $x\ne y$ then $\pi(n-e_x+e_{x-\nu'_j+\nu_j})=0$. Suppose towards a contradiction that $\pi(n-e_y+e_{y-\nu'_j+\nu_j})\ne0$ and that $\kappa_j\binom{y-\nu'_j+\nu_j}{\nu_j}\ne 0$. Then $\pi(n-e_y+e_{y-\nu'_j+\nu_j})\ne0$ implies that $\alpha(y-\nu'_j+\nu_j)\ne0$, and hence $P_\mu(y-\nu'_j+\nu_j,t)\ne0$ for some $t$. But this means that the state $y-\nu'_j+\nu_j$ is reachable for $\chem$ when started with initial distribution $\mu$. But $\kappa_j\binom{y-\nu'_j+\nu_j}{\nu_j}\ne 0$ implies that $y$ is reachable from $y-\nu'_j+\nu_j$ for $\chem$ via the $j$-th reaction, so we conclude that $y$ is reachable from $\mu$. But this implies that $P_\mu(y,t)\ne0$ for $t>0$, which in turn means that $\alpha(y)>0$. This contradicts our choice of $y$, so it must be that our assumption was wrong: either $\pi(n-e_y+e_{y-\nu'_j+\nu_j})=0$ or $\kappa_j\binom{y-\nu'_j+\nu_j}{\nu_j}=0$. But either of those imply the desired equality $\pi(n-e_y+e_{y-\nu'_j+\nu_j})(n_{y-\nu'_j+\nu_j}+1)\kappa_j\binom{y-\nu'_j+\nu_j}{\nu_j}=0$.
\end{itemize}
This proves that \eqref{eq:piQ=0} reduces to $0=0$ in this case. The reminder of the proof will be devoted to the second case; namely, the case where $n$ is such that $n_y=0$ for all $y\in\ZZ_{\ge0}^d$ with $\alpha(y)=0$.

Let $\XX=\{x\in\ZZ_{\ge0}^d:\alpha(x)\ne 0\}$. We claim that for every $x\notin\XX$ and every $j$, every summand in \eqref{eq:piQ=0} is zero. So fix $x\notin\XX$ and $j$:
\begin{itemize}
    \item $\pi(n-e_x)\kappa_I\mu(x)$: Since $\alpha(x)=0$, by choice of $n$ we have $n_x=0$. But this means that $n-e_x$ is negative at $x$ and hence $n-e_x\notin\CoarseStateSpace$, so $\pi(n-e_x)=0$.

    \item $\pi(n+e_x)\kappa_E (n_x+1)$: Notice that $\alpha(x)=0$ participates in the product defining $\pi(n+e_x)$, and hence $\pi(n+e_x)=0$.

    \item $\pi(n-e_x+e_{x-\nu'_j+\nu_j})(n_{x-\nu'_j+\nu_j}+1)\kappa_j\binom{x-\nu'_j+\nu_j}{\nu_j}$: As before, $n-e_x+e_{x-\nu'_j+\nu_j}\notin\CoarseStateSpace$ and hence $\pi(n-e_x+e_{x-\nu'_j+\nu_j})=0$.

    \item $\kappa_I\mu(x)$: Since $\alpha(x)=0$, it must be the case that $\mu(x)=0$, as otherwise $P_\mu(x,t)$ would be positive for sufficiently small $t$.

    \item $\kappa_En_x$: Since $\alpha(x)=0$, by choice of $n$ we have $n_x=0$.

    \item $n_x\kappa_j\binom{x}{\nu_j}$: Once again, $n_x=0$.
\end{itemize}
Thus we have shown that terms with $x\notin\XX$ do not contribute to \eqref{eq:piQ=0}. So to complete the proof, we have only to show that

\begin{align}\nonumber
    \sum_{x\in\XX}\Bigg[&\pi(n-e_x)\kappa_I\mu(x)+\pi(n+e_x)\kappa_E (n_x+1)\\\label{eq:sum-over-XX}
    &+\sum_j \pi(n-e_x+e_{x-\nu'_j+\nu_j})(n_{x-\nu'_j+\nu_j}+1)\kappa_j\binom{x-\nu'_j+\nu_j}{\nu_j}\Bigg]\\
    &\hspace{1.5in}=\pi(n)\sum_{x\in\XX}\left(\kappa_I\mu(x)+\kappa_En_x+\sum_j n_x\kappa_j\binom{x}{\nu_j}\right).\nonumber
\end{align}

Let $x\in\XX$ be arbitrary. Integration by parts gives
\begin{align*}
    \int_0^\infty\left(\frac{\mathrm d}{\mathrm dt}P_\mu(x,t)\right)\kappa_E e^{-\kappa_E t}\mathrm dt
    &=\kappa_E e^{-\kappa_E t}P_\mu(x,t)\Big|_{t=0}^{t=\infty}-\int_0^\infty P_\mu(x,t)(-\kappa_E^2 e^{-\kappa_E t})\mathrm dt\\
    &=-\kappa_E\mu(x)+\kappa_E\alpha(x).
\end{align*}
Because $P_\mu$ is the distribution for $\chem_\Rate$, the Kolmogorov forward equations for $\chem_\Rate$ tell us that
\begin{align*}
    \frac{\mathrm d}{\mathrm dt}P_\mu(x,t)
    &=\sum_{\nu_j\to \nu_j'}\kappa_j\binom{x-\nu'_j+\nu_j}{\nu_j}P_\mu(x-\nu_j'+\nu_j,t)
    -\sum_{\nu_j\to \nu_j'}\kappa_j\binom{x}{\nu_j}P_\mu(x,t)
\end{align*}
for each $t$. Plugging this in above and rearranging yields
\begin{align*}
    \sum_{\nu_j\to \nu_j'}\kappa_j\binom{x-\nu'_j+\nu_j}{\nu_j}\alpha(x-\nu_j'+\nu_j)
    -\sum_{\nu_j\to \nu_j'}\kappa_j\binom{x}{\nu_j}\alpha(x)
    &=-\kappa_E\mu(x)+\kappa_E\alpha(x)\\
    \kappa_E\frac{\mu(x)}{\alpha(x)}+\sum_{\nu_j\to \nu_j'}\kappa_j\binom{x-\nu'_j+\nu_j}{\nu_j}\frac{\alpha(x-\nu_j'+\nu_j)}{\alpha(x)}
    &=\kappa_E+\sum_{\nu_j\to \nu_j'}\kappa_j\binom{x}{\nu_j}.
\end{align*}
Note that we did not divide by zero in the second line because $\alpha(x)\ne0$ by definition of $\XX$. Since $x\in\XX$ was arbitrary, we can multiply through by $n_x$ and sum over $x$, which yields
\begin{align}\nonumber
    \sum_{x\in\XX}\bigg(n_x\kappa_E\frac{\mu(x)}{\alpha(x)}+n_x\sum_{\nu_j\to \nu_j'}\kappa_j\binom{x-\nu'_j+\nu_j}{\nu_j}&\frac{\alpha(x-\nu_j'+\nu_j)}{\alpha(x)}\bigg)\\\label{eq:above}
    &=\sum_{x\in\XX}\bigg(n_x\kappa_E+n_x\sum_{\nu_j\to \nu_j'}\kappa_j\binom{x}{\nu_j}\bigg).
\end{align}
Now we claim that $\mu$ and $\alpha$ are both probability measures supported on $\XX$. We know that $\mu$ is a probability measure by assumption; it is supported on $\XX$ because if $\mu(x)>0$ then $P_\mu(x,t)>0$ for small enough $t$ and hence $\alpha(x)>0$. We know that $\alpha$ is supported on $\XX$ by definition of $\XX$; to see that it is a probability measure, use the fact that the integrand in the definition of $\alpha$ is non-negative to interchange a sum over $x$ with the integral and then apply the fact that $P_\mu(x,t)$ is a probability measure for each $t$. Therefore $\mu$ and $\alpha$ are both probability measures supported on $\XX$, as claimed; it follows that $\sum_{x\in\XX} \kappa_I \mu(x)=\kappa_I=\sum_{x\in\XX}\kappa_I\alpha(x)$. So adding $\kappa_I$ to both sides of \eqref{eq:above} gives
\begin{align}\nonumber
    \sum_{x\in\XX}\bigg(n_x\kappa_E\frac{\mu(x)}{\alpha(x)}+\kappa_I\alpha(x)+n_x\sum_{\nu_j\to \nu_j'}\kappa_j&\binom{x-\nu'_j+\nu_j}{\nu_j}\frac{\alpha(x-\nu_j'+\nu_j)}{\alpha(x)}\bigg)\\\label{eq:almost done} &=\sum_{x\in\XX}\bigg(\kappa_I\mu(x)+n_x\kappa_E+n_x\sum_{\nu_j\to \nu_j'}\kappa_j\binom{x}{\nu_j}\bigg).
\end{align}
Now notice that, directly from the definition of $\pi$, we have
\begin{align*}
    \frac{\pi(n-e_x)}{\pi(n)}
    &=\frac{n_x}{\alpha(x)}\frac{\kappa_E}{\kappa_I}\\
    \frac{\pi(n+e_x)}{\pi(n)}
    &=\frac{\alpha(x)}{n_x+1}\frac{\kappa_I}{\kappa_E}\\
    \frac{\pi(n-e_x+e_{x-\nu_j'+\nu_j})}{\pi(n)}
    &=\frac{\alpha(x-\nu_j'+\nu_j)}{\alpha(x)}\frac{n_x}{n_{x-\nu_j'+\nu_j}+1},
\end{align*}
where the last equality holds for each reaction $\nu_j\to \nu_j'$. Applying these three in order on the left-hand side of \eqref{eq:almost done}, we get
\begin{align*}
    \sum_{x\in\XX}&\bigg(\kappa_I\mu(x)\frac{\pi(n-e_x)}{\pi(n)}+\kappa_E(n_x+1)\frac{\pi(n+e_x)}{\pi(n)}\\
    & +\sum_{\nu_j\to \nu_j'}\kappa_j\binom{x-\nu'_j+\nu_j}{\nu_j}(n_{x-\nu_j'+\nu_j}+1)\frac{\pi(n-e_x+e_{x-\nu_j'+\nu_j})}{\pi(n)}\bigg)\\
    &=\sum_{x\in\XX}\bigg(\kappa_I\mu(x)+n_x\kappa_E+n_x\sum_{\nu_j\to \nu_j'}\kappa_j\binom{x}{\nu_j}\bigg)\\
    \sum_{x\in\XX}&\bigg(\kappa_I\mu(x)\pi(n-e_x)+\kappa_E(n_x+1)\pi(n+e_x)\\
    &+\sum_{\nu_j\to \nu_j'}\kappa_j\binom{x-\nu'_j+\nu_j}{\nu_j}(n_{x-\nu_j'+\nu_j}+1)\pi(n-e_x+e_{x-\nu_j'+\nu_j})\bigg)\\
    &=\pi(n)\sum_{x\in\XX}\bigg(\kappa_I\mu(x)+n_x\kappa_E+n_x\sum_{\nu_j\to \nu_j'}\kappa_j\binom{x}{\nu_j}\bigg),
\end{align*}
which is exactly the desired equality, \eqref{eq:sum-over-XX}.
\end{proof}
\end{theorem}

Let us now consider some examples of applying this result.

\begin{example}\label{ex:birth death pi}
Let $\lambda\ge 0$, and consider the compartment system
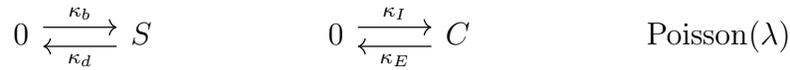
\begin{center}
\begin{tikzcd}
    0\arrow[yshift=.7ex]{r}{\kappa_b}& \arrow[yshift=-.7ex]{l}{\kappa_d}S & &
    0\arrow[yshift=.7ex]{r}{\kappa_I}& C\arrow[yshift=-.7ex]{l}{\kappa_E} & &
    \operatorname{Poisson}(\lambda)
\end{tikzcd}
\end{center}
Then the stationary distribution of the system is given by
\begin{align*}
\pi(n)=\left(\prod_{x=0}^\infty\frac{\alpha(x)^{n_x}}{n_x!}\right)\cdot\left[e^{-\kappa_I/\kappa_E}\cdot\left(\frac{\kappa_I}{\kappa_E}\right)^{\norm n_{\ell^1}}\right],
\end{align*}
where
\begin{align*}
\alpha(x)=\int_0^\infty \exp\left\{-(\lambda-\kappa_b/\kappa_d)e^{-\kappa_dt}-\kappa_b/\kappa_d\right\}\frac{((\lambda-\kappa_b/\kappa_d)e^{-\kappa_dt}+\kappa_b/\kappa_d)^x}{x!} \kappa_E e^{-\kappa_E t}\mathrm dt.
\end{align*}
\begin{proof}
Check that the distribution 
\begin{align*}
    P_\lambda(x,t):=\exp\left\{-(\lambda-\kappa_b/\kappa_d)e^{-\kappa_dt}-\kappa_b/\kappa_d\right\}\frac{((\lambda-\kappa_b/\kappa_d)e^{-\kappa_dt}+\kappa_b/\kappa_d)^x}{x!}
\end{align*}
is $\operatorname{Poisson}(\lambda)$ at time $t=0$ and satisfies
\begin{align*}
    \frac{\mathrm d}{\mathrm dt}P_\lambda(x,t)
    &=\kappa_b P_\lambda(x-1,t)+\kappa_d(x+1) P_\lambda(x+1,t)
    -\kappa_b P_\lambda(x,t)-\kappa_dx P_\lambda(x,t)
\end{align*}
for each $x$ and $t$, and apply Theorem \ref{Thm:Stationary Distribution}.
\end{proof}
\hfill $\triangle$
\end{example}

In the previous example, notice that the expected value of $\alpha$ is
\begin{align*}
    &\sum_{x=0}^\infty x\alpha(x)\\
    &=\int_0^\infty \sum_{x=0}^\infty\exp\left\{-(\lambda-\kappa_b/\kappa_d)e^{-\kappa_dt}-\kappa_b/\kappa_d\right\}\frac{((\lambda-\kappa_b/\kappa_d)e^{-\kappa_dt}+\kappa_b/\kappa_d)^{x+1}}{x!} \kappa_E e^{-\kappa_E t}\mathrm dt\\
    &=\int_0^\infty (\lambda-\kappa_b/\kappa_d)\kappa_Ee^{-(\kappa_d+\kappa_E)t}+\frac{\kappa_b\kappa_E}{\kappa_d}e^{-\kappa_E t}\mathrm dt\\
    &=\frac{(\lambda-\kappa_b/\kappa_d)\kappa_E}{\kappa_d+\kappa_E}+\frac{\kappa_b}{\kappa_d}\\
    &=\frac{\lambda\kappa_E+\kappa_b}{\kappa_d+\kappa_E}.
\end{align*}
This matches \cite{Duso_Zechner_2020}, where the same example is consider in section 2.A (see specifically their equation [20] and the following discussion). Note that in \cite{Duso_Zechner_2020}, though the expected value of $\alpha$ is calculated in general, an explicit formula for $\alpha(x)$ (in their notation, $P_\infty(x)$) is given in only two cases. The first is the case where $\lambda=\kappa_b/\kappa_d$, where (in section S7.4 of their SI Appendix) they remark that $\alpha$ is Poission with mean $\lambda$. This matches the formula we give above in Example \ref{ex:birth death pi}. The second case they cover is the one where $\kappa_d=0$. In that case they obtain
\begin{align*}
    \alpha(x)=(1-\xi)\xi^x e^{\lambda(1/\xi -1)}\frac{\Gamma(1+x,\lambda/\xi)}{x!},
\end{align*}
where $\xi=\kappa_b/(\kappa_b+\kappa_E)$ and $\Gamma$ is the upper incomplete Gamma function. One can check that this agrees with our next example, Example \ref{ex:birth only pi}, in the case where $\mu$ is taken to be Poission with parameter $\lambda$ by applying the binomial theorem in our formula and then making a change of variable in the integral.

The following example is interesting for a few reasons. First, the chemistry is not converging to any sort of stationary distribution, and yet the whole compartment model is. Second, notice that when $\mu$ is not a Poisson distribution, $P_\mu(x,t)$ is not a Poisson distribution in $x$ for all $t$ unlike the previous example or more generally the DR models from Remark \ref{rmk:DR}. Third, as discussed above, it generalizes an example from \cite{Duso_Zechner_2020}.

\begin{example}\label{ex:birth only pi}
Let $\mu$ be a probability distribution on $\ZZ_{\ge0}$, and consider the compartment system
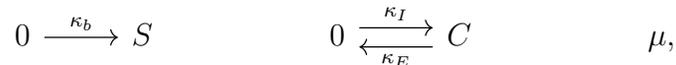
\begin{center}
\begin{tikzcd}
    0\arrow{r}{\kappa_b}& S& &
    0\arrow[yshift=.7ex]{r}{\kappa_I}& C\arrow[yshift=-.7ex]{l}{\kappa_E} & &
    \mu,
\end{tikzcd}
\end{center}
Then the stationary distribution of the system is given by
\begin{align*}
\pi(n)=\left(\prod_{x=0}^\infty\frac{\alpha(x)^{n_x}}{n_x!}\right)\cdot\left[e^{-\kappa_I/\kappa_E}\cdot\left(\frac{\kappa_I}{\kappa_E}\right)^{\norm n_{\ell^1}}\right],
\end{align*}
where
\begin{align*}
\alpha(x)=\int_0^\infty e^{-\kappa_b t}\left(\sum_{m=0}^x\frac{\kappa_b^mt^m}{m!}\mu(x-m)\right) \kappa_E e^{-\kappa_E t}\mathrm dt.
\end{align*}
\begin{proof}
Check that the distribution
\begin{align*}
    P_\mu(x,t):=e^{-\kappa_b t}\left(\sum_{m=0}^x\frac{\kappa_b^mt^m}{m!}\mu(x-m)\right)
\end{align*}
satisfies
\begin{align*}
    \frac{\mathrm d}{\mathrm dt}P_\mu(x,t)
    &=\kappa_b P_\mu(x-1,t)-\kappa_b P_\mu(x,t),
\end{align*}
with initial condition $P_\mu(x,0)=\mu(x)$, and apply Theorem \ref{Thm:Stationary Distribution}.
\end{proof}
\hfill $\triangle$
\end{example}

\section{Appendix}

The purpose of this section is to prove Theorem \ref{thm:lyapunov-transience}, which we recall says the following:

\begin{theorem}
    Let $X$ be a non-explosive continuous-time Markov chain on a countable discrete state space $\mathbb S$ with generator $\mathcal L$. Let $B\subset\mathbb S$, and let $\tau_B$ be the time for the process to enter $B$. Suppose there is some bounded function $V$ such that for all $x\in B^c$,
    \[
    \mathcal LV(x)\ge0.
    \]
    Then $\PP_{x_0}(\tau_B<\infty)<1$ for any $x_0$ such that
    \[
    \sup_{x\in B}V(x)<V(x_0).
    \]
\end{theorem}

Just like the theorem itself, the proof draws heavy inspiration from \cite{Stramer_Tweedie_1994}. Before providing the proof, we state the following well-known result:

\begin{lemma}[Dynkin's Formula]
Suppose $X$ is a Markov chain with finite state space $\mathbb S$, and let $\mathcal L$ be the generator of $X$. Then for any a.s. bounded stopping time $\tau$ and any $x\in\mathbb S$, we have 
\[
\EE_x[f(X_\tau)]=f(x)+\EE_x\left[\int_0^\tau \mathcal Lf(X_s)\mathrm ds\right]
\]
\end{lemma}

In fact Dynkin's Formula is well-known in much greater generality than what is stated above, but as stated it is not hard to prove and is enough for our purposes.

\begin{proof}[Proof of Theorem \ref{thm:lyapunov-transience}]
    Define $W$ on $\mathbb S$ via $W=V-\sup_{x\in B}V(x)$. Notice that $W(x_0)$ is strictly positive, $W$ is nonpositive on $B$, and $\mathcal LW=\mathcal LV$. Fix some enumeration of $\mathbb S$ in which $x_0$ is the first element, and for $m\in\NN$ let $\mathbb S_m$ denote the first $m$ elements of $\mathbb S$. Let $\tau_m$ be the first time $X$ is not in $\mathbb S_m$. Let $\Delta$ be a new state not in $\mathbb S$, and for $m\in\NN$ define a new Markov chain $X^m$ via
    \[
        X^m_t=
        \begin{cases}
            X_t     &   t<\tau_m\\
            \Delta  &   t\ge\tau_m
        \end{cases}
    \]
    Notice that $X^m$ has finite state space $\mathbb S_m\cup\{\Delta\}$. Notice that $W$ is bounded since $V$ is, let $C=\sup_{x\in\mathbb S} W(x)$, and extend $W$ to a function on $\mathbb S\cup\{\Delta\}$ by setting $W(\Delta)=C$. Let $\mathcal L_m$ denote the generator of the process $X^m$; we claim that $\mathcal LW(x)\le\mathcal L_mW(x)$ whenever $x\in\mathbb S_m$. Indeed, notice that
    \begin{align*}
        \EE_x[W(X_t)]&=\sum_{y\in\mathbb S}W(y)\PP_x(X_t=y)\\
        &=\sum_{y\in\mathbb S}W(y)\PP_x(X_t=y,t<\tau_m)+\sum_{y\in\mathbb S}W(y)\PP_x(X_t=y,t\ge \tau_m)\\
        &\le \sum_{y\in\mathbb S}W(y)\PP_x(X_t=y,t<\tau_m)+\sum_{y\in\mathbb S}C\PP_x(X_t=y,t\ge \tau_m)\\
        &=\sum_{y\in\mathbb S_m}W(y)\PP_x(X_t^m=y)+W(\Delta)\PP_x(X_t^m=\Delta)\\
        &=\EE_x[W(X^m_t)],
    \end{align*}
    and hence
    \[
    \mathcal LW(x)
    =\lim_{t\searrow0}\frac{\EE_x[W(X_t)]-W(x)}t
    \le\lim_{t\searrow0}\frac{\EE_x[W(X^m_t)]-W(x)}t
    =\mathcal L_mW(x),
    \]
    as claimed. Now for any $m$, applying Dynkin's Formula to the chain $X^m$ with finite stopping time $\tau_B\MIN\tau_m\MIN m$ yields
    \[
    \EE_{x_0}[W(X^m_{\tau_B\MIN\tau_m\MIN m})]=W(x_0)+\EE_{x_0}\left[\int_0^{\tau_B\MIN\tau_m\MIN m} \mathcal L_mW(X^m_s)\mathrm ds\right].
    \]
    But for $s<\tau_B\MIN\tau_m$ we have $X^m_s=X_s\in B^c\cap\mathbb S_m$ and hence
    \begin{align*}
        \mathcal L_mW(X^m_s)
        =\mathcal L_mW(X_s)
        \ge \mathcal LW(X_s)
        =\mathcal LV(X_s)
        \ge0.
    \end{align*}
    So the integrand in Dynkin's Formula is non-negative, and
    \begin{align*}
        W(x_0)
        &\le\EE_{x_0}[W(X^m_{\tau_B\MIN\tau_m\MIN m})]\\
        &=\EE_{x_0}[W(X^m_{\tau_B})\II_{\tau_B<\tau_m\MIN m}]+\EE_{x_0}[W(X^m_{\tau_m\MIN m})\II_{\tau_B\ge \tau_m\MIN m}]\\
        &\le\EE_{x_0}[W(X^m_{\tau_B})\II_{\tau_B<\tau_m\MIN m}]+C\PP_{x_0}(\tau_B\ge \tau_m\MIN m).
    \end{align*}
    Note that $X_{\tau_B}^m\in B$ on the event $\tau_B <\tau_m\wedge m$.  Hence $W(X_{\tau_B}^m) \II_{\tau_B<\tau_m\MIN m} \le 0$, and
    \[
    W(x_0)\le C\PP_{x_0}(\tau_B\ge \tau_m\MIN m)
    \]
    Since $X$ is assumed to be non-explosive, $\tau_m\to\infty$ as $m\to\infty$, so taking $m\to\infty$ above gives
    \[
    W(x_0)\le C\PP_{x_0}(\tau_B=\infty).
    \]
    But $W(x_0)$ is strictly positive and $0<W(x_0)\le C<\infty$, so $\PP_{x_0}(\tau_B = \infty) \ne 0$. That is, $\PP_{x_0}(\tau_B<\infty)<1$, as desired.
\end{proof}

\begin{remark}
    Note that the proof above gives us a lower bound for the probability that the process never returns to the set $B$:
    \[
    \frac{W(x_0)}{C}\le \PP_{x_0}(\tau_B=\infty),
    \]
    where $C = \sup_{x\in \mathbb S} W(x)$ and $W=V-\sup_{x\in B}V(x)$. We do not make use of this fact.
\end{remark}

\section{Acknowledgements}
We gratefully acknowledge support from NSF grant DMS-2051498 and Army Research Office grant W911NF-18-1-0324.
We  thank Daniele Cappelletti, Gheorghe Craciun, Erik Bates, and Evan Sorenson for some clarifying conversations, and John Sporel and Logan Heath who advised on notation.

\bibliography{June_2023_Revision}{}

\begin{thebibliography}{10}

\bibitem{Agbanusi_Isaacson_2014}
Ikemefuna~C. Agbanusi and Samuel~A. Isaacson.
\newblock A comparison of bimolecular reaction models for stochastic
  reaction–diffusion systems.
\newblock {\em Bulletin of Mathematical Biology}, 76, 2014.

\bibitem{Anderson_2007}
David~F. Anderson.
\newblock A modified next reaction method for simulating chemical systems with
  time dependent propensities and delays.
\newblock {\em J. Chem. Phys.}, 127(21):214107, 2007.

\bibitem{Anderson_Cappelletti_Kim_Nguyen_2020}
David~F. Anderson, Daniele Cappelletti, Jinsu Kim, and Tung~D. Nguyen.
\newblock Tier structure of strongly endotactic reaction networks.
\newblock {\em Stochastic Processes and their Applications},
  130(12):7218--7259, 2020.

\bibitem{Anderson_Higham_Leite_Williams_2019}
David~F. Anderson, Desmond~J. Higham, Saul~C. Leite, and Ruth~J. Williams.
\newblock On constrained langevin equations and (bio) chemical reaction
  networks.
\newblock {\em Multiscale Modeling \& Simulation}, 17(1):1--30, 2019.

\bibitem{Anderson_Kurtz_2015}
David~F. Anderson and Thomas~G. Kurtz.
\newblock {\em Stochastic Analysis of Biochemical Systems}.
\newblock Mathematical Biosciences Institute Lecture Series. Springer
  International Publishing, 2015.

\bibitem{Anderson_Schnoerr_Yuan_2020}
David~F. Anderson, David Schnoerr, and Chaojie Yuan.
\newblock Time-dependent product-form poisson distributions for reaction
  networks with higher order complexes.
\newblock {\em Journal of Mathematical Biology}, 80, May 2020.

\bibitem{anderson2017introduction}
David~F. Anderson, Timo Sepp{\"a}l{\"a}inen, and Benedek Valk{\'o}.
\newblock {\em Introduction to probability}.
\newblock Cambridge University Press, 2017.

\bibitem{Brauer_2008}
Fred Brauer.
\newblock Compartmental models in epidemiology.
\newblock In Fred Brauer, Pauline van~den Driessche, and Jianhong Wu, editors,
  {\em Mathematical Epidemiology}, pages 19--79. Springer Berlin Heidelberg,
  Berlin, Heidelberg, 2008.

\bibitem{Razo_Winkelmann_Klein_Hofling_2023}
Mauricio~J. del Razo, Stefanie Winkelmann, Rupert Klein, and Felix Höfling.
\newblock {Chemical diffusion master equation: Formulations of
  reaction–diffusion processes on the molecular level}.
\newblock {\em Journal of Mathematical Physics}, 64(1), 01 2023.
\newblock 013304.

\bibitem{Doi_1976}
Masao Doi.
\newblock Stochastic theory of diffusion-controlled reaction.
\newblock {\em Journal of Physics A: Mathematical and General}, 9(9):1479,
  September 1976.

\bibitem{Duso_Zechner_2020}
Lorenzo Duso and Christoph Zechner.
\newblock Stochastic reaction networks in dynamic compartment populations.
\newblock {\em Proceedings of the National Academy of Sciences},
  117(37):22674--22683, 2020.

\bibitem{Erban_Othmer_2014}
Radek Erban and Hans~G. Othmer.
\newblock Editorial: Special issue on stochastic modelling of
  reaction–diffusion processes in biology.
\newblock {\em Bulletin of Mathematical Biology}, 76:761--765, 2014.

\bibitem{Ethier_Kurtz_2009}
Stewart~N. Ethier and Thomas~G. Kurtz.
\newblock {\em Markov processes: characterization and convergence}.
\newblock John Wiley \& Sons, 2009.

\bibitem{Gibson_Bruck_2000}
Michael~A. Gibson and Jehoshua Bruck.
\newblock Efficient exact stochastic simulation of chemical systems with many
  species and many channels.
\newblock {\em J. Phys. Chem. A}, 105:1876--1889, 2000.

\bibitem{Gillespie_1976}
Daniel~T. Gillespie.
\newblock A general method for numerically simulating the stochastic time
  evolution of coupled chemical reactions.
\newblock {\em J. Comput. Phys.}, 22:403--434, 1976.

\bibitem{Gillespie_1977}
Daniel~T. Gillespie.
\newblock Exact stochastic simulation of coupled chemical reactions.
\newblock {\em J. Phys. Chem.}, 81(25):2340--2361, 1977.

\bibitem{Isaacson_2013}
Samuel~A. Isaacson.
\newblock A convergent reaction-diffusion master equation.
\newblock {\em The Journal of Chemical Physics}, 139(5):054101, 2013.

\bibitem{Jahnke_Huisinga_2007}
Tobias Jahnke and Wilhelm Huisinga.
\newblock Solving the chemical master equation for monomolecular reaction
  systems analytically.
\newblock {\em Journal of Mathematical Biology}, 54:1--26, 2007.

\bibitem{Kurtz_1978}
Thomas~G. Kurtz.
\newblock Strong approximation theorems for density dependent markov chains.
\newblock {\em Stochastic Processes and their Applications}, 6(3):223--240,
  1978.

\bibitem{Leite_Williams_2019}
Saul~C. Leite and Ruth~J. Williams.
\newblock A constrained langevin approximation for chemical reaction networks.
\newblock {\em The Annals of Applied Probability}, 29(3):1541--1608, 2019.

\bibitem{McKane_Newman_2004}
A.~J. McKane and T.~J. Newman.
\newblock Stochastic models in population biology and their deterministic
  analogs.
\newblock {\em Phys. Rev. E}, 70:041902, Oct 2004.

\bibitem{Meyn_Tweedie_1993}
Sean~P. Meyn and R.~L. Tweedie.
\newblock Stability of markovian processes iii: Foster–lyapunov criteria for
  continuous-time processes.
\newblock {\em Advances in Applied Probability}, 25(3):518–548, 1993.

\bibitem{Norris_1997}
James~R. Norris.
\newblock {\em Markov Chains}.
\newblock Cambridge Series in Statistical and Probabilistic Mathematics.
  Cambridge University Press, 1997.

\bibitem{Popovic_McKinley_Reed_2011}
Lea Popovic, Scott~A. McKinley, and Michael~C. Reed.
\newblock A stochastic compartmental model for fast axonal transport.
\newblock {\em SIAM Journal on Applied Mathematics}, 71(4):1531--1556, 2011.

\bibitem{Serfozo_2009}
Richard Serfozo.
\newblock {\em Basics of Applied Stochastic Processes}.
\newblock Probability and Its Applications. Springer-Verlag Berlin Heidelberg,
  2009.

\bibitem{Shinar_Feinberg_2010}
Guy Shinar and Martin Feinberg.
\newblock Structural sources of robustness in biochemical reaction networks.
\newblock {\em Science}, 327(5971):1389--1391, 2010.

\bibitem{Stramer_Tweedie_1994}
O.~Stramer and R.L. Tweedie.
\newblock Stability and instability of continuous-time markov processes.
\newblock In F.P. Kelly, editor, {\em Probability, Statistics and Optimization:
  A Tribute to Peter Whittle}, chapter~12, pages 173--184. Wiley Chichester,
  UK, 1994.

\bibitem{wilkinson2018stochastic}
Darren~J. Wilkinson.
\newblock {\em Stochastic modelling for systems biology}.
\newblock CRC press, 2018.

\bibitem{Xu_Hansen_Wiuf_2022}
Chuang Xu, Mads~Christian Hansen, and Carsten Wiuf.
\newblock Full classification of dynamics for one-dimensional continuous-time
  markov chains with polynomial transition rates.
\newblock {\em Advances in Applied Probability}, page 1–35, 2022.

\end{thebibliography}
\bibliographystyle{plain}

\end{document}